\documentclass[12pt]{article}

\usepackage{fullpage}
\usepackage{amsmath, amsthm, amsfonts, amssymb, amstext, mathrsfs, enumerate}
\usepackage{mathtools}
\usepackage{graphicx, ragged2e, lscape, framed, xcolor}
\usepackage{subfiles}
\usepackage{booktabs}

\theoremstyle{plain}
\newtheorem{theorem}{Theorem}[section]
\newtheorem{lemma}[theorem]{Lemma}

\newtheorem{conjecture}[theorem]{Conjecture}
\newtheorem{corollary}[theorem]{Corollary}

\newtheorem*{remark}{Remark}

\newtheorem{example}{Example}[section]

\numberwithin{equation}{section}
\allowdisplaybreaks

\newcommand{\affl}[3]{%
	\noindent #1,
	\textsc{#3}\\
	Email: \texttt{#2}\\[1.5pt]
}

\usepackage[pagebackref]{hyperref}
\hypersetup{
	colorlinks=true,
	urlcolor=purple,
	linkcolor=purple,
	citecolor=purple,
}

\DeclareMathOperator{\spec}{spec}

\title{\textbf{Spectral Radius, Vertex Deletion, and Chromatic Number of Signed Graphs}}

\author{Abhay Jayarajan, M. Rajesh Kannan, Priti Prasanna Mondal, \\ Shivaramakrishna Pragada}

\begin{document}
	\maketitle{}
	
	\begin{abstract}
		A signed graph $\Sigma=(G,\sigma)$ is a graph $G$ with edges given signs $1$ or $-1$ defined by the function $\sigma$. The adjacency matrix of $\Sigma$ is defined as per these signs. The relation between the largest eigenvalue of $G$ and $G-v$ has been studied in recent years, where $G-v$ is the graph obtained from $G$ by deleting the vertex $v$. In 2020, Sun and Das proved that the difference of the squares of the largest eigenvalues of the graphs $G$ and $G-v$ is bounded above by $2d(v)-1$ where $d(v)$ is the degree of $v$. A similar result need not be true for the largest eigenvalue of signed graphs. In this paper, we prove that the result is valid for the spectral radius of signed graphs. 
		
		On the other hand, the signed graph version of Hoffman's chromatic number bound was proved by Wang et al. in 2021. They also discussed the difficulty in proving the extended version encompassing all eigenvalues of $\Sigma$ as was done for unsigned graphs by Wocjan and Elphick. We note down a consequence of Wocjan and Elphick's lower bound for the chromatic number in terms of all the eigenvalues of $\Sigma_+$; all the eigenvalues of $\Sigma$ and $\Sigma_-$, where $\Sigma_+$ (resp. $\Sigma_-$) is the spanning subgraph induced by the positive (resp. negative) edges. We give examples where the result fails even under various restrictions on the signed graph. Finally, we improve an upper bound for the $k$-th power of the largest eigenvalue given by Stani\'{c} in terms of walks in signed graphs and give lower bounds for the least eigenvalue in terms of various parameters of $\Sigma_+$ and $\Sigma_-$.
	\end{abstract}
	
	
	
	\section{Introduction}
	A signed graph is a graph $\Sigma=(G,\sigma)$, where $G=(V, E)$ is the underlying graph, and $\sigma: E\rightarrow \{-1, +1\}$ is the sign function. Let $e_{ij}$ denote the edge between the vertices $v_{i}$ and $v_{j}$ of $G.$ Let $E^+(\Sigma)=\{e \in E: \sigma(e)=+1\}$ and $E^-(\Sigma)=\{e \in E: \sigma(e)=-1\}$. Let $\Sigma_+$ be the subgraph of $\Sigma$ with $V(\Sigma_+)=V(\Sigma),~ E(\Sigma_+)=E^+(\Sigma)$ and $\Sigma_-$ be the subgraph of $\Sigma$ with $V(\Sigma_-)=V(\Sigma),~ E(\Sigma_-)=E^-(\Sigma)$. Let $C$ be a cycle of signed graph $\Sigma$. the \emph{sign of a cycle} $C$ is denoted by $\sigma(C)$ and defined by $\sigma(C)=\prod_{e_{ij}\in C}\sigma(e_{ij}).$  A cycle is said to be \emph{positive} (respectively, \emph{negative}) if $\sigma(C)=+1$ (respectively, $\sigma(C)=-1$). A signed graph is said to be \emph{balanced} if all its cycles are positive. For a vertex $v \in V(G)$ define $=\{u\in V(G):\sigma(uv)=-1\} ~~~\mbox{and}~~~ N^{+}(v)=\{u\in V(G):\sigma(uv)=+1\}$. Denote by $N(v)$ the set of all vertices of $V(G)$ that are adjacent to $v$. Thus $N(v) = N^{+}(v) \cup N^{-}(v)$. Let $d(v) = \vert N(v) \vert$.
	The \emph{adjacency matrix} of a signed graph $\Sigma=(G,\sigma)$ denoted by  $A(\Sigma)=[a_{ij}^{\sigma}],$ is an $n \times n$ symmetric matrix, defined by:
	\[a_{ij}^{\sigma}=\begin{cases} 
		\sigma(e_{ij})  & \mbox{if}~ v_{i}\sim v_{j} \\
		0  & \mbox{otherwise}. 
	\end{cases}\]
	We denote by $\lambda_{1}(\Sigma)\geq\lambda_{2}(\Sigma)\geq \cdots\geq \lambda_{n}(\Sigma)$ the eigenvalues of $A(\Sigma)$. Let $\rho(\Sigma)=\max\{\lambda_{1}(\Sigma), -\lambda_{n}(\Sigma)\}$ denote the spectral radius of $A(\Sigma)$. Let $(G,1)$ or $(G,+)$ be a signed graph with the signs of all edges being positive. Then the eigenvalues of $(G,1)$ are ordered as $\lambda_{1}(G)\geq\lambda_{2}(G)\geq \cdots\geq \lambda_{n}(G)$ and the spectral radius is denoted by $\rho(G)$. We similarly define $(G,-1)$ or $(G,-)$. The graph $-\Sigma$ is the signed graph obtained from $\Sigma$ by reversing the signs on all the edges of $\Sigma$.
	
	A function $\eta: V(G)\rightarrow \{-1, +1\}$ is called a \emph{switching function}. Two signed graphs $\Sigma_{1}=(G, \sigma_{1})$ and $\Sigma_{2}=(G, \sigma_{2})$ are \emph{switching equivalent,} written as $\Sigma_{1}\sim \Sigma_{2},$ if there is a switching function $\eta: V(G)\rightarrow \{-1, +1\}$ such that \[\sigma_{2}(e_{ij})=\eta(v_{i})^{-1} \sigma_{1}(e_{ij})\eta(v_{j}).\]
	Suppose $S=\{v\in V(G): \eta(v)=1\}$; then one can see that $\Sigma_2$ turns out to be the graph obtained from $\Sigma_1$ by changing signs of all the edges in $\Sigma_1$ with one endpoint in $S$ and the other in its complement. In this case, we say that $\Sigma_2$ is the graph obtained by switching $\Sigma_1$ with respect to $S$. It is well known that switching equivalence preserves the spectrum of $\Sigma$.

	The largest eigenvalue of graphs has been extensively studied in the last few decades. Brualdi and Hoffman \cite{brualdi1985spectral} showed that for graphs with $m=\frac{k(k-1)}{2}$, $\lambda_1(G) \leq k-1$. This bound was improved by Stanley \cite{stanley1987bound} who found that $\lambda_1(G) \leq \frac{-1+\sqrt{1+8m}}{2}$. Later, Hong \cite{yuan1988bound} improved the bound for connected graphs and showed that $\lambda_1(G) \leq \sqrt{2m-n+1}$.  In 2019, Guo et al.\cite{guo2019sharp} gave another upper bound for the largest eigenvalue. In the paper, the authors proved that if $\frac{(k-2)(k-3)}{2}\leq m-n \leq \frac{k(k-3)}{2}$ where $3 \leq k \leq n$, then $\lambda_1(G) \leq \sqrt{2m-n-k+\frac{5}{2}+\sqrt{2m-2n+\frac{9}{4}}}$ with equality if and only if $G$ is a complete graph or $K_4-e$ which is the graph obtained from $K_4$ by removing an edge. In the same paper, the authors proposed a conjecture regarding the effect on the largest eigenvalue of a graph on removing a vertex and the edges that are incident to it. The authors conjectured that for any graph $G$ and a vertex $v \in V (G)$, \[\lambda_{1}(G) \leq \sqrt{\lambda_{1}^2(G-v)+2d(v)-1}.\] In 2020, Sun and Das \cite{sun2020conjecture} confirmed the conjecture and characterised all connected graphs for which this bound is attained. A short proof of this result was given by Jin et al.\cite{jin2024upper}, Liu and Ning \cite{MR5098803}. A related work for signed graphs has been done in \cite{xie2026index}. In \cite{lan2023remarks}, Lan et al. proved an analogue of this result in the case of signed graphs with some restrictions on the choice of vertex. However, this result is not generally true for signed graphs.  One of the main objectives of this paper is to show that the above bound for the largest eigenvalue remains valid when the largest eigenvalue is replaced by the spectral radius of a signed graph. In the route to proving this result, we also establish the equality conditions for the spectral radius inequality of block matrices by Jin et al. \cite{jin2024upper} in a special case.

	A proper $k$-coloring of an unsigned graph is a mapping $t: V(G) \rightarrow \{1, \dots, k \}$ such that no two adjacent vertices are mapped to the same integer. The least $k$ for which $G$ admits a proper $k$-coloring is called the chromatic number of $G$, denoted by $\chi(G)$. For $k\geq 1,$ let $M_{k}=\{ \pm 1, \pm 2, \cdots, \pm k/2\}$ if $k$ is even and $M_{k}=\{0, \pm 1, \pm 2, \cdots, \pm (k-1)/2\}$ if $k$ is odd. A (proper) $k$-coloring of a signed graph $\Sigma$ is a mapping $c: V(\Sigma) \rightarrow M_{k}$ such that $c(u)\neq \sigma(uv)c(v)$ for every edge $uv \in E(\Sigma)$. The signed chromatic number of $\Sigma$, denoted by $\chi(\Sigma)$, is the minimum number $k$ for which $\Sigma$ admits a $k$-coloring\cite{mavcajova2014chromatic}. The \emph{Hoffman lower bound} of chromatic number \cite{hoffman2003eigenvalues} $$\chi(G) \geq 1-\dfrac{\lambda_1(G)}{\lambda_n(G)}$$ is a well-known result in spectral graph theory. A generalization of \emph{Hoffman's lower bound} for chromatic number for an unsigned graph was given by Wocjan and Elphick \cite{wocjan2013new}, which comprised all the eigenvalues of the graph.
	
	Recently, Wang et al. \cite{wang2021eigenvalues} studied the signed version of \emph{Hoffman's lower bound} for chromatic number and discussed the difficulty of Wocjan and Elphick's bound for signed graphs. In this paper, in Section \ref{section_chromatic}, we give an analogous bound for signed graphs that comprises all the eigenvalues of the $\Sigma_+$; and all eigenvalues of signed graph as well as the eigenvalues of $\Sigma_{-}$, where, as defined earlier, $\Sigma_{+}$ (resp $\Sigma_{-}$)  is the spanning subgraph induced by the positive (resp negative) edges. The counterexample for Wocjan-Elphick type bound in signed graphs given by Wang et al. has $-\lambda_n > \lambda_1$. A natural question to pose is whether this $\lambda_1$ attaining the spectral radius gives a good strength for the signed graphs for the bound to hold. It turns out that even $\lambda_1 > \vert \lambda_n \vert$ in addition to connectedness of signed graphs cannot necessarily make a signed graph to force the Wocjan-Elphick type bound hold. We give examples for these claims in this section.\\
	\newline
	In 2019, Stani\'{c} \cite{Stanic2019bounding} obtained an upper bound for the square of the largest eigenvalue. We recover the bound as a consequence of a norm inequality. Stani\'{c} also gave a bound for the $k$th power of the largest adjacency eigenvalue of a signed graph in  terms of the number of positive and negative walks of length $k$ starting at a vertex. We give an improved bound in Section \ref{Signed_Adjacency}. In \cite{Stanic2018perturbations}, Stani\'{c} gave a tool to obtain
	lower bounds for the largest eigenvalue of signed graph in terms of various parameters
	of $\Sigma_+$ and $\Sigma_-$. Using a  similar strategy, we give lower bounds for the least eigenvalue of
	signed graphs and illustrate an example where the extremal graphs can be obtained.
	
	
	\section{Spectral Radius of Block Matrices}
	
	Let $A$ be a symmetric matrix partitioned as
	\[
	A=
	\begin{pmatrix}
		0 & A_{12}\\
		A_{21} & A_{22}
	\end{pmatrix}.
	\]
	Set
	\[
	\rho_2=\rho(A_{22}),\qquad
	\mu=\rho(A_{12}A_{22}A_{21}),\qquad
	\nu=\rho(A_{12}A_{21})
	\]
	where $\rho(Z)$ stands for the spectral radius of a matrix $Z$ which is the maximum of the moduli of the eigenvalues of $Z$.
	
	Consider the polynomial
	\[
	g(x)=(x-\nu)^2(\rho_2^2+x)-\mu^2.
	\]
	Since $g'(x)>0$ for all $x>\nu$, the function $g$ is strictly increasing on $[\nu,\infty)$. Furthermore, $g(\nu)=-\mu^2\le 0.$
	Hence, if $\theta^*$ denotes the largest real root of $g$, then $\theta^*\ge \nu$. Moreover, $\theta^*=\nu$ if and only if $\mu=0.$
	
	\begin{lemma}\label{Spec_Rad_Sym1}\emph{\cite{jin2024upper}}
		Let $A$ be a symmetric matrix partitioned as
		\[
		A=
		\begin{pmatrix}
			0 & A_{12}\\
			A_{21} & A_{22}
		\end{pmatrix}.
		\]
		Then
		\[
		\rho(A)^2\le \rho_2^2+\theta^*,
		\]
		where $\rho_2=\rho(A_{22})$, $\mu=\rho(A_{12}A_{22}A_{21})$, $\nu=\rho(A_{12}A_{21})$, and $\theta^*$ is the largest real root of
		\[
		g(x)=(x-\nu)^2(\rho_2^2+x)-\mu^2.
		\]
	\end{lemma}
	Jin et al. \cite{jin2024upper} discussed the equality case for $\theta^* > \nu$ in case of nonnegative matrices. We discuss vertex deletion in graphs in next section and hence we are particularly interested in the case when the block $A_{21}$ is a vector. Therefore, let $A$ be partitioned as follows:
	\[
	A=
	\begin{pmatrix}
		0 & a^T\\
		a & B
	\end{pmatrix}.
	\]
	Set
	\[
	\rho_{2}:=\rho(B),\qquad
	\nu:=a^{T}a,\qquad
	m:=a^{T}Ba,\qquad
	\mu:=|m|
	\]
	
	Let $\{w_i\}_{i=1}^{n-1}$ be an orthonormal basis of $\mathbb{R}^{n-1}$ consisting of eigenvectors of $B$. Then
	\[
	B=\sum_{i=1}^{n-1} \mu_i w_iw_i^T
	\]
	and write
	\[
	a=\sum_{i=1}^{n-1} c_iw_i,\qquad
	c_i=w_i^Ta,\qquad
	\nu=\sum_{i=1}^{n-1} c_i^2,\qquad
	m=\sum_{i=1}^{n-1} c_i^2\mu_i.
	\]

	We first need the following lemma. The proof of the lemma is routine, so we skip.
	
	\begin{lemma}\label{equivalence}
		$B^2a=\rho_2^2\,a$ if and only if $c_i = 0$ whenever $\mu_i^2\ne\rho_2^2$.
	\end{lemma}
		
		
		
		
		
	We now prove the sufficiency for the equality case of Lemma \ref{Spec_Rad_Sym1} when $A_{21}$ is restricted to a vector.
	\begin{theorem}\label{necessity}
		If 
		\begin{equation}\label{eq:condC}
			B^2a=\rho_2^2\,a,
		\end{equation} then  $\rho(A)^2=\rho_2^2+\theta^* $.
	\end{theorem}
	\begin{proof}
		If $a=0$ the result is trivial. Let  $B^2a=\rho_2^2\,a,$ and $a \neq 0$. Suppose $\rho_2 \neq 0$. By Lemma \ref{Spec_Rad_Sym1}, we already have that,
		\begin{equation}\label{eq:JZZ}
			\rho(A)^2 \;\le\; \rho_2^2+\theta^*
		\end{equation}
		Write $\nu_+=\sum_{\mu_i=\rho_2}c_i^2$, \quad
		$\nu_-=\sum_{\mu_i=-\rho_2}c_i^2$, so by Lemma \ref{equivalence}, $\nu=\nu_++\nu_-$ and
		$m=(\nu_+-\nu_-)\rho_2$.
		
		\smallskip
		Let $\mu > 0$ (i.e.\ $\theta^*>\nu$). Define
		\[
		\lambda:=\frac{m}{\theta^*-\nu}.
		\]
		Since $g(\theta^*)=0$, we get, $$m^2=(\theta^*-\nu)^2(\rho_2^2+\theta^*)$$ So
		\begin{align*}
			\lambda^2&=\dfrac{m^2}{(\theta^*-\nu)^2}=\dfrac{(\theta^*-\nu)^2(\rho_2^2+\theta^*)}{(\theta^*-\nu)^2}\\
			\lambda^2&=\rho_2^2+\theta^*>\rho_2^2
		\end{align*}
		Therefore, $|\lambda|>\rho_2$, so
		$\lambda I-B$ is invertible. 
		\newline
		Set $y_1:=(\lambda I-B)^{-1}a$ and
		$y:=(1, y_1^T)^T$. By construction, $ay_0+By_1=\lambda y_1$ (with $y_0=1$), which is
		the bottom block of $Ay=\lambda y$. \\
		Next, note that \\
		$$
		Bw_i=\mu_iw_i \implies (\lambda I-B)^{-1}
		w_i=\dfrac{1}{\lambda-\mu_i}w_i$$
		Therefore, we get that,
		\begin{align*}
			y_1&= (\lambda I-B)^{-1}(\sum_{i=1}^{n-1}c_iw_i)\\
			&= \sum_{i=1}^{n-1}c_i(\lambda I-B)^{-1}w_i\\
			y_1&= \sum_{i=1}^{n-1}c_i\dfrac{1}{\lambda-\mu_i}w_i
		\end{align*}
		Hence we have,
		\begin{equation}
			a^Ty_1= \bigg(\sum_{j=1}^{n-1}c_jw_j^T\bigg)\bigg(\sum_{i=1}^{n-1}\dfrac{c_i}{\lambda-\mu_i}w_i\bigg)
		\end{equation}
		So,
		\[
		a^Ty_1=\sum \limits_{i=1}^{n-1} \frac{c_i^2}{\lambda-\mu_i}
		=\frac{\nu_+}{\lambda-\rho_2}+\frac{\nu_-}{\lambda+\rho_2}
		=\frac{\lambda\nu+\rho_2(\nu_+-\nu_-)}{\lambda^2-\rho_2^2}
		=\frac{\lambda\nu+m}{\theta^*}
		=\frac{\lambda\nu+\lambda(\theta^*-\nu)}{\theta^*}=\lambda,
		\]
		using $\lambda^2-\rho_2^2=\theta^*$ and $m=\lambda(\theta^*-\nu)$ . This is the top block of $Ay=\lambda y$
		\medskip
		
		Let $\mu=0$ (i.e.\ $\theta^{*}=\nu$).
		Then $m=0$. Define
		\[
		\lambda:=\sqrt{\rho_2^{\,2}+\nu}>\rho_2
		\]
		and again
		\[
		y_1:=(\lambda I-B)^{-1}a.
		\]
		
		The same computation gives $ay_0+By_1=\lambda y_1$ (with $y_0=1$) and 
		\[
		a^{T}y_1
		=
		\frac{\lambda\nu+m}{\lambda^{2}-\rho_2^{\,2}}
		=
		\frac{\lambda\nu+0}{\nu}
		=
		\lambda,
		\]
		So in both cases,
		\[
		y=(1,y_1^T)^T
		\]
		is an eigenvector for $\lambda$.
		Thus, $A$ has an eigenvalue of modulus
		\[
		\sqrt{\rho_2^{\,2}+\theta^{*}},
		\]
		so
		\[
		\rho(A)^2\ge \rho_2^{\,2}+\theta^{*};
		\]
		
		Using \eqref{eq:JZZ}, equality follows. 
		
		If $\rho_2 = 0$ then $B=0$ forcing $m= \mu =0$ and $\theta^*=\nu$.
		Since $a \neq 0$ put $ \lambda= \sqrt{\nu}$, $y=(1, a^T/\lambda)^T$. Then 
		\begin{equation*}
			Ay= \begin{pmatrix}
				a^Ta/\lambda\\
				a 
			\end{pmatrix}=\begin{pmatrix}
				\lambda\\
				a 
			\end{pmatrix}= \lambda y
		\end{equation*}
		Therefore $\rho(A)^2 \geq \nu= \rho_2^2+ \theta^*$ and the equality follows by \eqref{eq:JZZ}.
	\end{proof}

	\begin{theorem}\label{lem:outer}
		If $\rho(A)^2=\rho_2^2+\theta^* $ then 
		\begin{equation}\label{eqn2}
			B^2a=\rho_2^2\,a,
		\end{equation}
	\end{theorem}
	
	\begin{proof}
		If $a=0$ the result is trivial. Let $a \neq 0$.
		
		Suppose $\rho(A)^2=\rho_2^2+\theta^*$. Replacing $(a,B)$ by $(-a,-B)$
		replaces $A$ by $-A$, since
		\[
		\begin{pmatrix}0&-a^T\\-a&-B\end{pmatrix}=-A
		\]
		exactly; this leaves $\nu,\rho_2,\mu$
		unchanged (so $\theta^*$ is unchanged), and preserves \eqref{eqn2}
		(since $(-B)^2(-a)=\rho_2^2(-a) \iff B^2a=\rho_2^2a$). Since $\rho(-A)=\rho(A)$,
		we may assume without loss of generality after performing this
		substitution if needed that $\rho(A)$ is attained at a \emph{positive}
		eigenvalue. That is, if $-\rho(A)$ was the extremal eigenvalue of $A$,
		then $+\rho(A)$ is the extremal eigenvalue of $-A$, and proving
		\eqref{eqn2} for the transformed pair proves it for the original pair. 
		
		So let $\lambda:=\rho(A)>0$ with unit eigenvector $y=(y_0,  y_1^T)^T$ where $y_1 \in \mathbb{R}^{n-1}$. Let $y_1= \sum \limits_{i=1}^{n-1}d_iw_i$ Since $\rho(A)^2=\rho_2^2+\theta^*>\rho_2^2$, $\lambda>\rho_2$, so $\lambda$ is
		not an eigenvalue of $B$.\\
		If $y_0=0$ then the lower block of the eigenvalue equation $Ay=\lambda y$ gives 
		$By_1=\lambda y_1$ which is not possible by the above argument. So $y_0\ne0$. Without loss of generality let $y_0 > 0$.
		
		The lower block of $Ay=\lambda y$ gives $ay_0+By_1=\lambda y_1$, that is, $(\lambda I-B)y_1=ay_0$. Since $\lambda$ is not an eigenvalue of $B$, we can write $y_1=(\lambda I-B)^{-1}ay_0$. Now since $Bw_i = \mu_i w_i$ we must have $(\lambda I-B)^{-1}w_i= \frac{1}{\lambda- \mu_i}w_i$. Therefore,\\
		\begin{align*}
			y_1&=(\lambda I-B)^{-1}y_0 \sum \limits_{i=1}^{n-1} c_i w_i\\
			&=y_0\sum \limits_{i=1}^{n-1}c_i(\lambda I-B)^{-1}w_i \\
			&=\sum \limits_{i=1}^{n-1}\dfrac{y_0 c_i}{\lambda- \mu_i}w_i \\
		\end{align*}
		Therefore, $d_i=\dfrac{y_0 c_i}{\lambda- \mu_i}$. Also, the first block of $A y= \lambda y$ gives $\lambda y_0=a^Ty_1$. But $a^Ty_1=\sum \limits_{i=1}^{n-1}c_id_i=\sum \limits_{i=1}^{n-1}\dfrac{y_0 c_i^2}{\lambda- \mu_i}$. Therefore, dividing by $y_0 \neq 0$, we get, 
		
		\begin{equation}\label{eq:secular2}
			\lambda=\sum \limits_{i=1}^{n-1}\frac{c_i^2}{\lambda-\mu_i}.
		\end{equation}
		Since $\lambda>\rho_2\ge|\mu_i|$, $\forall i \in \{1,\dots,n-1\}$,
		$\lambda\mp\mu_i>0$, and $(\lambda-\mu_i)(\lambda+\mu_i)=\lambda^2-\mu_i^2\ge\lambda^2-\rho_2^2>0$. Therefore,\\
		\begin{align*}
			\dfrac{1}{\lambda-\mu_i}&\leq \dfrac{\lambda+\mu_i}{\lambda^2-\rho_2^2}\\
			\dfrac{c_i^2}{\lambda-\mu_i}&\leq \dfrac{c_i^2(\lambda+\mu_i)}{\lambda^2-\rho_2^2}\\
			\sum \limits_{i=1}^{n-1}\dfrac{c_i^2}{\lambda-\mu_i}&\leq \sum \limits_{i=1}^{n-1}\dfrac{c_i^2(\lambda+\mu_i)}{\lambda^2-\rho_2^2}\\
			\lambda &\leq \sum \limits_{i=1}^{n-1}\dfrac{c_i^2(\lambda+\mu_i)}{\lambda^2-\rho_2^2}\\
		\end{align*}
		Define $R:= \dfrac{\lambda \nu+m}{\lambda^2-\rho_2^2}$. Then \\
		$\lambda \nu+m=\sum \limits_{i=1}^{n-1}\lambda c_i^2+\sum \limits_{i=1}^{n-1}c_i^2 \mu_i=\sum \limits_{i=1}^{n-1}c_i^2(\lambda+\mu_i)$. Therefore, $$R= \sum \limits_{i=1}^{n-1}\dfrac{c_i^2(\lambda+\mu_i)}{\lambda^2-\rho_2^2}$$
		Therefore, $\lambda \leq R$. Note that the equality $R =\lambda$ holds if and only if  $\mu_i^2=\rho_2^2$ whenever $c_i \neq 0$. Therefore, it is enough to show that $R= \lambda$ if $\rho(A)^2=\rho_2^2+\theta^*$.

		
		Since $g(\theta^*)=0$, $m^2=(\theta^*-\nu)^2\lambda^2$, i.e.\
		$m=\pm\lambda(\theta^*-\nu)$.
		\begin{itemize}
			\item If $\theta^*=\nu$ (i.e.\ $\mu=0$, so $m=0$), we directly get that,
			$R=(\lambda\nu+0)/\nu=\lambda$.
			\item If $\theta^*>\nu$, we already have $R\ge\lambda$; since
			$R(\lambda^2-\rho_2^2)=\lambda\nu+m$ and $\lambda^2-\rho_2^2=\theta^*$, together gives that $\lambda\nu+m\ge\lambda\theta^*$, i.e.\ $m\ge\lambda(\theta^*-\nu)$.\\
			Now suppose
			$m=-\lambda(\theta^*-\nu)$ this would force
			$-\lambda(\theta^*-\nu)\ge\lambda(\theta^*-\nu)$, i.e.\
			$2\lambda(\theta^*-\nu)\le0$, impossible as $\lambda,\theta^*-\nu>0$; so
			$m=\lambda(\theta^*-\nu)$, and substituting back, we get,
			$R=(\lambda\nu+\lambda(\theta^*-\nu))/\theta^*=\lambda\theta^*/\theta^*=\lambda$.
		\end{itemize}
		In both cases $R=\lambda$, so  $\mu_i^2=\rho_2^2$ whenever $c_i \neq 0$,
		forcing \eqref{eqn2} by the equivalence established above.
	\end{proof}
	
	We combine the two theorems above and write the following result.
	\begin{theorem}\label{matrix_eq}
		Under the notations fixed above, $\rho(A)^2=\rho_2^2+\theta^* $ if and only if 
		$B^2a=\rho_2^2\,a$.
		
	\end{theorem}
	
	\section{Signed Spectral Radius}
	Guo et al.\cite{guo2019sharp} presented the effect on the largest eigenvalue of a graph when a vertex is deleted and proposed the following conjecture. 
	\begin{conjecture}\emph{\cite{guo2019sharp}}
		Let $G$ be a connected graph and $v$ be any vertex in $G$. Then 
		$$ \lambda_{1}(G) \leq \sqrt{\lambda_{1}^2(G-v)+2d(v)-1}$$
	\end{conjecture}
	Sun and Das \cite{sun2020conjecture} gave a proof of this conjecture and also characterized the connected graphs for which the bound is attained.
	\begin{theorem}\emph{\cite{sun2020conjecture}}\label{conjec_result}
		Let $G$ be a connected graph and $v$ be any vertex in $G$. Then 
		$$ \lambda_{1}(G) \leq \sqrt{\lambda_{1}^2(G-v)+2d(v)-1}$$
		Equality holds if and only if $G$ is a star graph and $v$ is a pendant vertex, or $G$ is complete.
	\end{theorem}
	A short proof of this result was given by Jin et al.\cite{jin2024upper}, Liu and Ning \cite{MR5098803}. In \cite{lan2023remarks}, an analogue of this result was proven with some restrictions for a signed graph $\Sigma$. The authors proved the following:
	\begin{theorem}\emph{\cite{lan2023remarks}}\label{Remarks_result}
		Let $\Sigma$ be a connected signed graph. Let $v$ be a vertex such that $\Sigma-v$ is balanced. Then $$\lambda_1(\Sigma)\leq \sqrt{\lambda_1^2(\Sigma-v)+2d(v)-1}$$
	\end{theorem}
	However, the result is not true in general for signed graphs \cite{lan2023remarks}. A simple example is the following: a signed graph $\Sigma$ obtained from an unbalanced triangle with one negative edge and a pendant vertex attached to the vertex in the triangle having $2$ positive neighbors; the result does not hold when $v$ is chosen to be the pendant vertex. Nevertheless, in this section, we prove that the perturbation result is true for the spectral radius of signed graphs. To this end, we recall the following general inequality between spectral radius of symmetric matrix and its principal submatrix.
	
	\begin{lemma}\label{Spec_Rad_Sym}\emph{\cite{jin2024upper}}
		Let $A$ be a symmetric matrix partitioned as
		\[
		A=
		\begin{pmatrix}
			0 & A_{12}\\
			A^T_{12} & A_{22}
		\end{pmatrix}.
		\]
		Then
		\[
		\rho(A)^2\le \rho_2^2+\theta^*,
		\]
		where $\rho_2=\rho(A_{22})$, $\mu=\rho(A_{12}A_{22}A^T_{12})$, $\nu=\rho(A_{12}A^T_{12})$, and $\theta^*$ is the largest real root of
		\[
		g(x)=(x-\nu)^2(\rho_2^2+x)-\mu^2.
		\]
	\end{lemma}
	
	Finally, we also need the characterization of spectral radius of a symmetric matrix in terms of Rayleigh quotient.
	
	\begin{lemma}\label{Spec_Rad_Rayleigh}
		Let $\Sigma=(G,\sigma)$ be a signed graph and let $0\neq \mathbf{x}\in\mathbb{R}^n$. Then
		\[
		\rho(\Sigma)\ = \max_{\mathbf{x} \neq 0} \left|
		\frac{\mathbf{x}^{T}A(\Sigma)\mathbf{x}}
		{\mathbf{x}^{T}\mathbf{x}}
		\right|.
		\]
	\end{lemma}
	
	For ease of reading the main theorem and proofs, we introduce the following notation. Let $\Sigma=(G,\sigma)$ be a signed graph and let $v\in V(G)$ be a
	vertex with degree $d(v)$. Let $q_v(\Sigma)=t_v^+(\Sigma)-t_v^-(\Sigma)$ be the difference of number of positive triangles $t_v^+(\Sigma)$ in $\Sigma$ containing $v$ and the number of negative triangles $t_v^-(\Sigma)$ in $\Sigma$ containing $v$. Define 
	\[g_v(x):=(x-d(v))^2(\rho^2(\Sigma-v)+x) -4\vert q_v(\Sigma) \vert ^2.\]
	
	Now, we are ready to state one of our main results of this section. 
	\begin{theorem}\label{Spec_Rad_Theta}
		Let $\Sigma=(G,\sigma)$ be a signed graph and let $v$ be a non-isolated vertex. Let $\theta^{*}$ denote the largest real root of $g_v(x)$.
		Then
		\[
		\rho^{2}(\Sigma)\leq \rho^{2}(\Sigma-v)+\theta^{*}.
		\]
		Furthermore, if $t_v^+=t_v^-$, then
		\[
		\rho^{2}(\Sigma)\leq \rho^{2}(\Sigma-v)+d(v).
		\]
	\end{theorem}
	
	Before giving the proof of theorem, we quickly note some properties of the $g_v(x)$ which prove useful later on. Observe that $g'_v(x)>0$ for all $x>d(v)$, the function $g_v$ is strictly increasing on $[d(v),\infty)$. Furthermore, $g(d(v))=-4\vert q_v(\Sigma)\vert^2\le 0.$
	Hence, if $\theta^*$ denotes the largest real root of $g_v$, then $\theta^*\ge d(v)$. Hence
	$\theta^{*}$ is that unique root in $[d(v),\infty)$. Moreover, $\theta^*=d(v)$ if and only if $q_v(\Sigma)=0.$ 
	
	\begin{proof}
		Let $\Sigma=(G,\sigma)$ be a signed graph and let $v\in V(G)$ be a
		vertex with $d(v)=d$. Switch $\Sigma$ with respect to
		$N^{-}(v)$ so that, in the resulting signed graph $\Sigma'$, every
		edge incident with $v$ is positive;
		\[
		\rho(\Sigma')=\rho(\Sigma)
		\quad\text{and}\quad
		\rho(\Sigma'-v)=\rho(\Sigma-v).
		\]
		Order the vertices of $\Sigma'$ with $v$ first and write
		\[
		A:=A(\Sigma')
		=
		\begin{pmatrix}
			0 & a^{T}\\
			a & B
		\end{pmatrix},
		\qquad
		B:=A(\Sigma'-v),
		\]
		where $a\in\{0,1\}^{\,n-1}$ is the indicator vector of $N(v)$.
		Note that,\\
		\[
		a^{T}Ba
		=2\bigl(|E^{+}(\Sigma'[N(v)])|-|E^{-}(\Sigma'[N(v)])|\bigr)=2(t_v^+(\Sigma')-t_v^-(\Sigma'))
		\] 
		Since the sign of every cycle is switching invariant, we must have,
		\[
		a^{T}Ba
		=2(t_v^+(\Sigma)-t_v^-(\Sigma))=2q_v(\Sigma)
		\] 
		Thus we have $\rho(a^TBa)= \vert a^TBa\vert=2\vert q_v(\Sigma) \vert$, $\rho(a^Ta)=d(v)$. So by Lemma \ref{Spec_Rad_Sym}, we have, \[
		\rho^{2}(\Sigma)\leq \rho^{2}(\Sigma-v)+\theta^{*}.
		\]
		If $t_v^+=t_v^-$, then the largest real root of
		$g_v(x)$ is $\theta^{*}=d(v)$. Hence,
		\[
		\rho^{2}(\Sigma)
		\leq \rho^{2}(\Sigma-v)+d(v). \qedhere
		\]
	\end{proof}
	
	Next, we prove a general lower bound on spectral radius of signed graph which is useful in proving the main result and its equality case. 
	
	\begin{lemma}\label{Spec_Rad_lower}
		For $n \geq 2$,
		\[\rho^2(\Sigma)\geq \dfrac{4(\vert E^+(\Sigma)\vert-\vert E^-(\Sigma)\vert)^2}{(n-1)^2}-(2n-1).\]
		Equality holds if and only if $\Sigma\cong(K_n,+)$ or $\Sigma \cong(K_n,-)$.
	\end{lemma}
	\begin{proof}
		Applying Lemma \ref{Spec_Rad_Rayleigh} on $\mathbf{x}=\mathbf{1}$, where $\mathbf{1}$ is the vector of all-ones,  we get,      \begin{equation}\label{spec_all_one_bound} 
			\rho^2(\Sigma)\geq \bigg(\dfrac{\textbf{1}^TA(\Sigma)\textbf{1}}{\textbf{1}^{T}\textbf{1}}\bigg)^2=\dfrac{4(\vert E^+(\Sigma)\vert-\vert E^-(\Sigma)\vert)^2}{n^2}.
		\end{equation} 
		It is easy to see that 
		\begin{equation}\label{max_edges}
			2\left\vert \vert E^+(\Sigma)\vert-\vert E^-(\Sigma)\vert\right\vert\leq 2(\vert E^+(\Sigma)\vert+\vert E^-(\Sigma)\vert)\leq n(n-1).
		\end{equation}
		Hence,
		\begin{align*}
			&\dfrac{4(\vert E^+(\Sigma)\vert-\vert E^-(\Sigma)\vert)^2}{n^2}-\Bigg(\dfrac{4(\vert E^+(\Sigma)\vert-\vert E^-(\Sigma)\vert)^2}{(n-1)^2}-(2n-1)\Bigg)\\
			&=(2n-1)\left(1-\dfrac{4\left(\vert E^+(\Sigma)\vert-\vert E^-(\Sigma)\vert\right)^2}{n^2(n-1)^2}\right)\\
			&\geq 0
		\end{align*}
		Substituting the above inequality in Equation (\ref{spec_all_one_bound}) gives the desired result.
		
		Clearly for $\Sigma\cong(K_n,+)$ and $\Sigma\cong(K_n,-)$ the equality holds. Conversely for the equality to hold the last equality must hold, which in turn forces the equality to hold at both the places in inequality (\ref{max_edges}). That is true when $\vert E^+(\Sigma)\vert=0$ or $\vert E^-(\Sigma)\vert=0$ which means $\Sigma\cong(K_n,+)$ or $\Sigma \cong(K_n,-)$.
	\end{proof}
	
	Now, we are ready to state our main theorem.
	
	\begin{theorem}\label{Spec_Rad_Bound1}
		Let $\Sigma=(G,\sigma) $ be a signed graph and $v$ be a non-isolated vertex.  Then \[\rho(\Sigma) \leq \sqrt{\rho^2(\Sigma-v)+2d(v)-1}.\]
		Let $\Sigma^\prime$ be the signed graph obtained from  $\Sigma$ by switching with respect to $N^{-}(v)$. Let $d=d(v) \ge 2$, then the equality holds if and only if
		\begin{itemize}
			\item[(a)] $\Sigma^\prime[N(v)]\cong (K_d,+)$ or
			$\Sigma^\prime[N(v)]\cong (K_d,-)$; and
			\item[(b)] $\rho(\Sigma-v)=d-1$.
		\end{itemize}
	\end{theorem}
	\begin{proof}
		Let $\Sigma^\prime$ be the signed graph obtained from $\Sigma$ by switching with respect to the set $N^-(v)$. Then the adjacency matrix of $\Sigma^\prime$ can be written in the block form
		\[
		A(\Sigma^\prime)=
		\begin{pmatrix}
			0 & a^T\\
			a & B
		\end{pmatrix},
		\]
		where $B=A(\Sigma^\prime-v)$ and $a$ is the column vector whose entries indicate the neighbours of $v$ in $\Sigma^\prime-v$.
		
		Let $\theta^*$ be the largest real root of $g_v(x)$.
		If $d(v)=1$, then $t_v^+=t_v^-=0$, and hence we must have
		\[
		g_v(x)=(x-d(v))^2\bigl(\rho^{2}(\Sigma-v)+x\bigr).
		\]
		Since $g_v(d(v))=0$, $\theta^*=d(v)$, and the result follows immediately.
		
		So assume that $d(v)\ge 2$. Since $2d(v)-1>d(v)$, Theorem~\ref{Spec_Rad_Theta} implies that it suffices to show that $g_v(2d(v)-1)\ge 0.$
		We have
		\begin{align}
			\rho^2(\Sigma-v)
			&=\rho^2(\Sigma^\prime-v) \notag\\
			&\ge \rho^2(\Sigma^\prime[N(v)]) \label{Nbr_Ineq_1} \\
			&\ge
			\frac{4\bigl(|E^{+}(\Sigma^\prime[N(v)])|-|E^{-}(\Sigma^\prime[N(v)])|\bigr)^2}
			{(d(v)-1)^2}
			-(2d(v)-1) \label{Nbr_Ineq_2}\\
			&=
			\frac{4\vert q_v(\Sigma)\vert^2}{(d(v)-1)^2}
			-(2d(v)-1)\notag,
		\end{align}
		where the second inequality follows from Lemma~\ref{Spec_Rad_lower}. The last equality follows from the fact that 
		\[
		|E^{+}(\Sigma^\prime[N(v)])|-|E^{-}(\Sigma^\prime[N(v)])|
		= q_v(\Sigma')=q_v(\Sigma).
		\]
		
		Multiplying the above inequality by $(d(v)-1)^2$ yields
		\[
		(d(v)-1)^2\bigl(\rho^2(\Sigma-v)+2d(v)-1\bigr)
		\ge 4\vert q_v(\Sigma)\vert^2.
		\]
		Since
		\[
		g_v(2d(v)-1)
		=(d(v)-1)^2\bigl(\rho^2(\Sigma-v)+2d(v)-1\bigr)
		-4\vert q_v(\Sigma) \vert^2,
		\]
		and $d(v)-1\ge 1$, it follows that
		\[g_v(2d(v)-1)\ge 0.\]
		Therefore, $\theta^*\le 2d(v)-1$. Combining this with Theorem~\ref{Spec_Rad_Theta} completes the proof of the inequality. 
		
		For the equality, let $d(v) \geq 2$. Suppose first that
		\[ \rho^2(\Sigma)=\rho^2(\Sigma-v)+2d(v)-1. \]
		From the proof of the inequality,
		\[ \rho^2(\Sigma) \leq \rho^2(\Sigma-v)+\theta^* \leq \rho^2(\Sigma-v)+2d(v)-1. \]
		Hence $\theta^*=2d(v)-1$, and therefore
		\[ g_v(2d(v)-1)=0. \]
		It follows that
		\[ \rho^2(\Sigma-v) = \frac{4\vert q_v(\Sigma)\vert^2}{(d(v)-1)^2} -(2d(v)-1).\]
		Together with \eqref{Nbr_Ineq_1} and \eqref{Nbr_Ineq_2}, this gives
		\[ \rho^2(\Sigma-v) =\rho^2(\Sigma^\prime-v) \geq \rho^2(\Sigma^\prime[N(v)])
		\geq \frac{4\vert q_v(\Sigma)\vert^2}{(d(v)-1)^2} -(2d(v)-1) =\rho^2(\Sigma-v).\]
		Thus equality holds in both \eqref{Nbr_Ineq_1} and \eqref{Nbr_Ineq_2}. By Lemma~\ref{Spec_Rad_lower}, condition (a) follows. Equality in \eqref{Nbr_Ineq_1} then gives
		\[ \rho^2(\Sigma-v) =\rho^2(\Sigma^\prime[N(v)]) =(d(v)-1)^2, \]
		and hence condition (b) follows.
		
		Conversely, suppose that conditions (a) and (b) hold. By condition (b) and switching invariance,
		\[ \rho(B)=\rho(\Sigma^\prime-v)=\rho(\Sigma-v)=d(v)-1. \]
		
		First suppose that $\Sigma^\prime[N(v)]\cong(K_d,+)$. Then
		\[ a^TBa=d(v)(d(v)-1) \quad\text{and}\quad a^Ta=d(v). \]
		
		Thus the Rayleigh quotient of $a$ for $B$ is $d(v)-1=\rho(B)$. Since $B$ is symmetric, equality in the Rayleigh quotient gives
		\[ Ba=(d(v)-1)a. \]
		Consequently,
		\[ 
		A(\Sigma^\prime)
		\begin{pmatrix}
			1\\
			a
		\end{pmatrix}
		=
		\begin{pmatrix}
			a^Ta\\
			a+Ba
		\end{pmatrix}
		=d(v)
		\begin{pmatrix}
			1\\
			a
		\end{pmatrix}. \]
		Hence $d(v)$ is an eigenvalue of $A(\Sigma^\prime)$.
		
		Now suppose that $\Sigma^\prime[N(v)]\cong(K_d,-)$. Then
		\[ a^TBa=-d(v)(d(v)-1) \quad\text{and}\quad a^Ta=d(v). \]
		Thus the Rayleigh quotient of $a$ for $B$ is $-(d(v)-1)$, whose absolute value is $\rho(B)$. Since $B$ is symmetric, equality in the Rayleigh quotient gives
		\[ Ba=-(d(v)-1)a. \]
		Consequently,
		\[
		A(\Sigma^\prime)
		\begin{pmatrix}
			-1\\
			a
		\end{pmatrix}
		=
		\begin{pmatrix}
			a^Ta\\
			-a+Ba
		\end{pmatrix}
		=-d(v)
		\begin{pmatrix}
			-1\\
			a
		\end{pmatrix}.\]
		Hence $-d(v)$ is an eigenvalue of $A(\Sigma^\prime)$.
		
		Therefore, in either case,
		\[ \rho(\Sigma)=\rho(\Sigma^\prime)\geq d(v). \]
		On the other hand, the inequality already proved and condition (b) give
		\[ \rho(\Sigma) \leq \sqrt{(d(v)-1)^2+2d(v)-1} =d(v).\]
		Thus
		\[\rho^2(\Sigma)=\rho^2(\Sigma-v)+2d(v)-1,\]
		which completes the equality case.
	\end{proof}
	
	\begin{theorem}
		Let $d(v)=1$. Then equality in Theorem \ref{Spec_Rad_Bound1} holds if and only if $\rho(\Sigma-v) = \sqrt{d(u)-1}$ where $u$ is the unique neighbor of $v$ in $\Sigma$.
	\end{theorem}
	\begin{proof}
		Since $d(v)=1$, order the unique neighbor $u$ of $v$ first so that the vector $a$ can be chosen to be the vector $e_1$ which is the vector with 1 in the first coordinate and zero elsewhere. Denote by $\rho$, $\rho_2$ the spectral radii of $A(\Sigma')$, $A(\Sigma'-v)$ respectively. \\
		Suppose that the equality $\rho^2(\Sigma) = \rho^2(\Sigma-v) +1$ holds. Then by Theorem \ref{matrix_eq}, $B^2 a = \rho_2 ^2 a$. The first coordinates of L.H.S. and R.H.S. of this equation yield the required condition.

		Conversely, suppose $\rho_2^2 = d(u)-1$. Then
		\begin{equation*}
			A^2(\Sigma^\prime)=
			\begin{pmatrix}
				e_1^Te_1 & e_1^TB\\
				Be_1 & e_1 e_1^T+B^2
			\end{pmatrix}
		\end{equation*}
		By Cauchy Interlacing Theorem,
		$\rho^2 \geq \lambda_1(e_1 e_1^T+B^2)$. The Rayleigh Quotient inequality gives,
		\begin{align*}
			\lambda_1(e_1 e_1^T+B^2) &\geq \dfrac{e_1^T(e_1 e_1^T+B^2)e_1}{e_1^T e_1}\\
			&=\dfrac{e_1^Te_1 e_1^Te_1+e_1^TB^2e_1}{1}\\
			&=1+d_{\Sigma-v}(u)=d(u)
		\end{align*}
		Hence we obtain $\rho^2 \geq d(u)= \rho_2^2 +1$. We already have $\rho^2 \leq d(u)= \rho_2^2 +1$ from Theorem \ref{Spec_Rad_Bound1}. Hence the equality holds.
	\end{proof}
	
	The following result can be derived from the above Theorem~\ref {Spec_Rad_Bound1}, which improves Theorem~\ref {Remarks_result} in (\cite{lan2023remarks}, Theorem 1.10).
	\begin{corollary}\label{Remarks_result_coro}
		Let $\Sigma$ be a connected signed graph. Let $v$ be a vertex such that $|\lambda_{n-1}(\Sigma-v)|\leq \lambda_{1}(\Sigma-v) $. Then $$\lambda_1(\Sigma)\leq \sqrt{\lambda_1^2(\Sigma-v)+2d(v)-1}.$$
	\end{corollary}
	
	From the above Corollary \ref{Remarks_result_coro}, one can get the Theorem \ref{conjec_result} in \cite{sun2020conjecture}.
	\begin{corollary}
		Let $G$ be a connected graph and $v$ be any vertex in $G$. Then 
		$$ \lambda_{1}(G) \leq \sqrt{\lambda_{1}^2(G-v)+2d(v)-1}.$$
	\end{corollary}
	
	\begin{remark}\emph{
			The following are a few examples where the conditions in Theorem \ref{Spec_Rad_Bound1} are satisfied and hence the equality in the bound is attained.
			\begin{enumerate}
				\item Let $\Sigma$ be a signed graph switching equivalent to $(K_n,+)$ that is $\Sigma$ is a complete balanced graph on $n$ vertices. Then $\rho(\Sigma)=n-1$ and for any $v \in V(K_n)$, we have $\rho(\Sigma-v)=n-2$. 
				\item Let $\Sigma$ be a signed graph switching equivalent to $(K_n,-)$-the complete graph on $n$ vertices with all edges given negative signs. Here, again $\rho(\Sigma)=n-1$ and $\rho(\Sigma-v)=n-2$ for any vertex $v \in V(K_n)$.
				\item Let $n \geq 4$ and  $\Sigma_n$ denote a signed graph switching equivalent to the signed graph obtained from $(K_{n-1},+)$ and vertex $u$ by joining $u$ with  $v_1$ by a positive edge and with $v_2$ by a negative edge, where $v_1, v_2$ are two arbitrary vertices of the complete graph $(K_{n-1},+)$. Choose $v$ to be a vertex in the clique different from $v_1, v_2$. It can be seen that $\spec(\Sigma_n)=\{(n-2)^{(1)},1^{(1)},-1^{(n-3)},-2^{(1)}\}$. Clearly for $n \geq 5$, $\Sigma_n-v$ is switching equivalent to $\Sigma_{n-1}$. Thus, $\sqrt{\rho^2(\Sigma_n-v)+2d(v)-1}=\sqrt{(n-3)^2+2(n-2)-1}=n-2=\rho(\Sigma_n)$
		\end{enumerate}}
	\end{remark}
	
	Finally, we give another upper bound for $\rho(\Sigma)$ which might, at times, beat the above bound.
	\begin{theorem}\label{Spec_Rad_Bound2}
		Let $\Sigma$ be a signed graph and $v$ be a non-isolated vertex. Then $$ \rho(\Sigma) \leq \sqrt{\rho^2(\Sigma-v)+d(v)+\dfrac{2\vert q_v(\Sigma) \vert}{\sqrt{\rho^2(\Sigma-v)+d(v)}}}$$
	\end{theorem}
	\begin{proof} Let us consider the function $g_v(x)$ at $x_0= d(v)+\dfrac{2\vert q_v(\Sigma) \vert}{\sqrt{\rho^2(\Sigma-v)+d(v)}}$, we get
		\begin{align*}
			&~~~~~~~~ g\Bigg(d(v)+\dfrac{2\vert q_v(\Sigma) \vert}{\sqrt{\rho^2(\Sigma-v)+d(v)}}\Bigg)\\
			&=\Bigg(\dfrac{2\vert q_v(\Sigma) \vert}{\sqrt{\rho^2(\Sigma-v)+d(v)}}\Bigg)^2\Bigg(\rho^2(\Sigma-v)+d(v)+\dfrac{\vert 2q_v (\Sigma) \vert}{\sqrt{\rho^2(\Sigma-v)+d(v)}}\Bigg) -4(\vert q_v(\Sigma) \vert)^2\\
			&=\dfrac{8\vert q_v(\Sigma) \vert^3}{(\rho^2(\Sigma-v)+d(v))^{\frac{3}{2}}}\geq 0
		\end{align*}
		Since $x_0 \geq d(v)$ and $g(x)$ is increasing on $[d(v), \infty)$, we must have $\theta^* \leq x_0$.
		Applying Theorem \ref{Spec_Rad_Theta} gives the required inequality.
	\end{proof}
	
	\begin{remark}
		Sometimes Theorem \ref{Spec_Rad_Bound2} gives a better bound for the spectral radius compared to Theorem \ref{Spec_Rad_Bound1}. For instance, with a vertex $v$ with $d(v) \geq 2$ and $t_v^+=t_v^-$, we have $ \rho^2(\Sigma) \leq \rho^2(\Sigma-v)+d(v)+\dfrac{2\vert q_v(\Sigma) \vert}{\sqrt{\rho^2(\Sigma-v)+d(v)}}<\rho^2(\Sigma-v)+2d(v)-1$. 
	\end{remark}
	
	\section{Signed Chromatic Number }\label{section_chromatic}
	The well-known Hoffman's lower bound on the chromatic number $\chi(G)$ of a graph is stated as follows:
	\begin{theorem}\emph{\cite{hoffman2003eigenvalues}}\label{Hoffman_bound}
		For a non-empty graph $G$ on $n$ vertices $$\chi(G) \geq 1 +\dfrac{\lambda_1(G)}{-\lambda_n(G)}.$$
	\end{theorem}
	Let $x=(x_{1}, x_{2},\cdots, x_{n} )$ be an element of $\mathbb{R}^{n}.$ Let $x^{\downarrow}$ be the vector obtained by rearranging the coordinates of $x$ in the non-increasing order. Thus, if $x^{\downarrow}=(x_{1}^{\downarrow}, x_{2}^{\downarrow},\cdots, x_{n}^{\downarrow}),$ then  $x_{1}^{\downarrow}\geq x_{2}^{\downarrow}\geq\cdots\geq x_{n}^{\downarrow}.$ 
	
	Let $x, y \in \mathbb{R}^{n}$. The vector $x$ is majorized by $y,$ in symbols $x \prec y,$ if \[ \sum_{i=1}^{m}x_{i}^{\downarrow} \leq \sum_{i=1}^{m}y_{i}^{\downarrow} ~~~~\mbox{for}~~ m=1, 2, \cdots, n-1; ~~\mbox{and}~~ \sum_{i=1}^{n}x_{i}^{\downarrow} = \sum_{i=1}^{n}y_{i}^{\downarrow}.\]

	For a matrix \(B\), let \(B^{*}\) denote the conjugate transpose of $B$. Let $A$ be a Hermitian matrix of order $n$, and  $\lambda_{1}\geq \lambda_{2}\geq \cdots\geq \lambda_{n}$ be  the eigenvalues of $A$. Consider the vector $\lambda^{\downarrow}(A)=(\lambda_{1}, \lambda_{2}, \cdots, \lambda_{n})$ and $\lambda^{\uparrow}(A)=(\lambda_{n}, \lambda_{n-1}, \cdots, \lambda_{1})$. Then the eigenvalues of any two Hermitian matrices satisfy the following majorization.
	\begin{lemma}\emph{\cite{bhatia2013matrix}}\label{majorization_+}
		If $A,B$ are Hermitian matrices, then the eigenvalues of $A, B$ and $A + B$ satisfy the following majorization relation:
		\[\lambda^{\downarrow}(A)+\lambda^{\uparrow}(B)\prec \lambda^{\downarrow}(A+B)\prec \lambda^{\downarrow}(A)+\lambda^{\downarrow}(B) \]
	\end{lemma}
	Wocjan and Elphick \cite{wocjan2013new} proved the following results and used the majorization technique to obtain a bound on the chromatic number.
	\begin{theorem}\emph{\cite{wocjan2013new}}\label{diag_unitary} Let $G$ be a simple unsigned graph, and there exists a coloring of $G$ with $c$ colors. Then, there exist $c$ diagonal unitary matrices $U_{1}, U_{2},\cdots, U_{c} $ whose entries are $c^{th}$ roots of unity such that \[\sum_{s=1}^{c}U_{s}(W\odot A)U_{s}^{*}=0,\]
		where $W$ is an arbitrary self-adjoint matrix and $W\odot A$ denotes the entry-wise product of $W$ and $A.$
	\end{theorem}
	\begin{corollary}\emph{\cite{wocjan2013new}}\label{unitary_corollary}
		Assume that there exists a coloring of $G$ with $c$ colors. Then, there exist $c-1$ diagonal unitary matrices $U_{1}, U_{2},\cdots, U_{c-1} $ whose entries are $c^{th}$ roots of unity such that \[\sum_{s=1}^{c-1}U_{s}(-W\odot A)U_{s}^{*}=W\odot A.\]
	\end{corollary}
	\begin{theorem}\emph{\cite{wocjan2013new}}\label{encompassing}
		Let $G$ be a simple non-empty graph with $n$ vertices. Let $\lambda_{1}(G)\geq\lambda_{2}(G)\geq \cdots \geq \lambda_{n}(G) $ be the eigenvalues of $A(G).$ Then
		\[\chi(G)\geq 1+  \max_{m=1, 2, \cdots, n-1} \frac{\sum_{i=1}^{m}\lambda_{i}(G)}{-\sum_{i=1}^{m}\lambda_{n-i+1}(G)}, \] where $\chi(G)$ is the chromatic number of $G$.
	\end{theorem}
	
	In 2021, Wang et al. \cite{wang2021eigenvalues} proved Hoffman's lower bound version for signed graphs. 
	\begin{theorem}\emph{\cite{wang2021eigenvalues}}\label{Hoffman_signed}
		For a non-empty signed graph $\Sigma$ on $n$ vertices, $$\chi(\Sigma)\geq 1+ \dfrac{\lambda_1(\Sigma)}{-\lambda_n(\Sigma)}.$$
		where $\chi(\Sigma)$ denotes the signed chromatic number of $\Sigma$.
	\end{theorem}
	
	One can observe that restricting $m=1$ in Theorem \ref{encompassing} gives Hoffman's bound. Wang et al., in their work, noted that such an improved bound is not directly possible in the case of signed graphs. Let $\Sigma=(K_n,-1)$ be the complete graph on $n \geq 3$ vertices with each of the edges given negative signs. In this case, for $m=n-1$, 
	$$1+   \frac{\sum_{i=1}^{m}\lambda_{i}(\Sigma)}{-\sum_{i=1}^{m}\lambda_{n-i+1}(\Sigma)} = 1+\frac{-\lambda_{n}(\Sigma)}{\lambda_{1}(\Sigma)}=1+\dfrac{n-1}{1}=n > \chi(\Sigma) = 2$$

	For an unsigned graph $G$ and a subgraph $H,$ the chromatic number satisfies the inequality $\chi(G)\geq \chi (H).$ The analogous result for signed graph is as follows:
	\begin{lemma}\label{Chromatic_lemmma}
		Let $\Sigma$ be a signed graph and $\Sigma_{1}$ be a subgraph of $\Sigma.$ Then $\chi(\Sigma)\geq \chi(\Sigma_{1}).$
	\end{lemma}
	The following result gives a lower bound on the chromatic number of signed graph.
	
	\begin{theorem}\label{Sigma+_Bound}
		Let $\Sigma$ be a signed graph with $n$ vertices and at least one positive edge. Let $\lambda_{1}(\Sigma_+)\geq\lambda_{2}(\Sigma_+)\geq \cdots \geq \lambda_{n}(\Sigma_+) $ be the eigenvalues of $A(\Sigma_+)$. Then \begin{equation}
			\chi(\Sigma)\geq 1+ \max_{m=1, 2, \cdots, n-1} \frac{\sum_{i=1}^{m}\lambda_{i}(\Sigma_+)}{-\sum_{i=1}^{m}\lambda_{n-i+1}(\Sigma_+)}.
		\end{equation} 
	\end{theorem}
	\begin{proof}
		The signed graph $\Sigma_{+}$ is a subgraph with all positive edges. By Lemma \ref{Chromatic_lemmma} and Theorem \ref{encompassing} we get,\\
		\begin{equation*}
			\chi(\Sigma) \geq \chi(\Sigma_+)\geq 1+ \max_{m=1, 2, \cdots, n-1} \frac{\sum_{i=1}^{m}\lambda_{i}(\Sigma_+)}{-\sum_{i=1}^{m}\lambda_{n-i+1}(\Sigma_+)}.
		\end{equation*}
	\end{proof}
	
	The following result is a consequence of Theorem \ref{Sigma+_Bound}. In cases where the eigenvalues of $\Sigma_-$ are easy to compute, the bound becomes handy. 
	
	\begin{corollary}\label{Chroma_Treom}
		Let $\Sigma$ be a non-empty signed graph with $n$ vertices and at least one positive edge. Let $B=-A(\Sigma_-)$ and let $\lambda_{1}(\Sigma)\geq\lambda_{2}(\Sigma)\geq \cdots \geq \lambda_{n}(\Sigma) $ be the eigenvalues of $A(\Sigma)$ and $\lambda_{1}(B)\geq\lambda_{2}(B)\geq \cdots \geq \lambda_{n}(B) $ be the eigenvalues of $B.$ Then \begin{equation}\label{signed_hoff_gene}
			\chi(\Sigma)\geq 1+ \max_{m=1, 2, \cdots, n-1} \frac{\sum_{i=1}^{m}\lambda_{i}(\Sigma)+\sum_{i=1}^{m}\lambda_{n-i+1}(B)}{-\sum_{i=1}^{m}\lambda_{n-i+1}(\Sigma)-\sum_{i=1}^{m}\lambda_{n-i+1}(B)}.
		\end{equation} 
	\end{corollary}
	\begin{proof}
		Let $A= A(\Sigma)$, $B=-A(\Sigma_-)$, $C=A(\Sigma_+)$. Then clearly $C=A+B$.

		For a symmetric matrix $X$ of order $n$ define:
		\[
		P_m(X)=\sum_{i=1}^{m}\lambda_i(X),
		\qquad
		Q_m(X)=
		-\sum_{i=1}^{m}\lambda_{n-i+1}(X),
		\]
		where $1 \leq m < n.$ If $C\neq 0$, both $P_m(C)$ and $Q_m(C)$ are positive. The quotient in
		(\ref{signed_hoff_gene}) is
		\[
		R_m=
		\frac{P_m(A)-Q_m(B)}
		{Q_m(A)+Q_m(B)}.
		\]
		
		By Lemma \ref{majorization_+}, we get
		\begin{equation}\label{P,Q_Ineq}
			P_m(C)\geq P_m(A)-Q_m(B),
			\qquad
			Q_m(C)\leq Q_m(A)+Q_m(B).
		\end{equation}
		
		If $P_m(A)-Q_m(B)$ is nonnegative, then (\ref{P,Q_Ineq}) implies
		\[
		R_m\leq
		\frac{P_m(C)}
		{Q_m(C)}.
		\]
		
		If $P_m(A)-Q_m(B)$ is negative, the same conclusion is immediate because
		the quotient on the right is positive. Therefore
		\[
		1+\max_m R_m
		\leq
		1+\max_m
		\frac{P_m(C)}
		{Q_m(C)}.
		\]
		By Theorem \ref{Sigma+_Bound},
		$$\chi(\Sigma) \geq \chi(\Sigma_+) \geq 1+\max_m
		\frac{P_m(C)}
		{Q_m(C)} \geq 1+ \max_m R_m
		.
		$$
		
	\end{proof}

	\begin{remark}\label{Chrom_Num_Gen}
		The signed graph $\Sigma$ in Theorem \ref{Sigma+_Bound} and Corollary \ref{Chroma_Treom} can be replaced by any signed graph that is switching equivalent to it. Hence for a signed graph if $T(\Sigma)= \{ \Sigma^{'} \sim \Sigma: \Sigma^{'} \text{has at least one positive edge} \}$
		$$\chi(\Sigma)\geq  1+ \max_{\Sigma^\prime \in T(\Sigma)}~\max_{m=1, 2, \cdots, n-1} \frac{\sum_{i=1}^{m}\lambda_{i}(\Sigma^{'}_+)}{-\sum_{i=1}^{m}\lambda_{n-i+1}(\Sigma^\prime_+)}$$
		and $$\chi(\Sigma)\geq  1+ \max_{\Sigma^\prime \in T(\Sigma)}~\max_{m=1, 2, \cdots, n-1} \frac{\sum_{i=1}^{m}\lambda_{i}(\Sigma)+\sum_{i=1}^{m}\lambda_{n-i+1}(-\Sigma^\prime_-)}{-\sum_{i=1}^{m}\lambda_{n-i+1}(\Sigma)-\sum_{i=1}^{m}\lambda_{n-i+1}(-\Sigma^\prime_-)}.$$
	\end{remark}

	We construct a few examples of signed graphs for which the bound in Corollary \ref{Chroma_Treom} is better than the bound in Theorem \ref{Hoffman_signed}. 
	\begin{example}
		Take two copies of $K_{15}$ with all positive edges. Let $u_1$ be a fixed vertex in one of the copies and $v_1, v_2$ be two fixed vertices in another copy. Let $\Sigma$ be the graph obtained from these two copies of $K_{15}$ by joining $u_1$ and $v_1$ by a positive edge and $u_1$ and $v_2$ by a negative edge. Then in the notation of Corollary \ref{Chroma_Treom}, $\Sigma_- \cong (K_2,-) \cup 28K_1$. Hence $\spec(B) =\spec(-A(\Sigma_-))=\{1^{(1)},0^{(28)},-1^{(1)}\}$. Also, it can be shown that $-1$ is an eigenvalue of $\Sigma$ with multiplicity $26$, and the remaining eigenvalues are the roots of the equation $(x-14)(x^{3}-12x^{2}-29x+12)=0$. Also one can check that $14.008< \lambda_1(\Sigma) < 14.00896$, $\lambda_2(\Sigma)=14$, $0.36138< \lambda_3(\Sigma) < 0.37$, $-2.37034< \lambda_{30}(\Sigma) < -2.37033$. This gives that, $\dfrac{\lambda_{1}}{-\lambda_{30}}< \dfrac{14.00896}{2.37033}<6$. Hence, Wang et al.'s bound cannot be above $7$. On the other hand, $\dfrac{\lambda_{1}(\Sigma)+\lambda_{2}(\Sigma)+\lambda_{30}(B)+\lambda_{29}(B)}{-\lambda_{30}(\Sigma)-\lambda_{29}(\Sigma)-\lambda_{30}(B)-\lambda_{29}(B)}=\dfrac{\lambda_{1}(\Sigma)+13}{-\lambda_{30}(\Sigma)+2}> \dfrac{14.008+13}{2.37034+2}>6$. Hence, the bound in Corollary \ref{Chroma_Treom} gives a bound which is above $7$.
	\end{example}
	\begin{example}\label{circulant_exam}
		A Circulant graph $C_{n}\{a_{1},a_{2},\dots,a_{k}\}$ is a graph with $V(G)=\mathbb{Z}_{n}$. For $v_{i}, v_{j} \in V(G)$, $v_{i} \sim v_{j}$ if and only if $ v_{i} - v_{j} \in \{a_{1},a_{2},\dots, a_{k},-a_1,-a_{2},\dots, -a_{k}\} ( \mod n )$. Consider the Circulant graph $G\cong C_{10}\{2,3,5\}$. Define $\sigma: E(G) \rightarrow\{1,-1\}$ as\\
		\[\sigma(e_{ij})=\begin{cases} 
			1  & \mbox{if}~ v_i-v_j\in \{2,-2,3,-3\} (\mod 10) \\
			-1  & \mbox{if}~ v_i-v_j\in \{5,-5\} (\mod 10) \\  
		\end{cases}\]
		Let $\Sigma=(G, \sigma)$ in Figure \ref{fig:circulant}. Then $\spec(\Sigma) =\{3^{(1)},1^{(5)},(\sqrt{5}-2)^{(2)},(-\sqrt{5}-2)^{(2)}\}$. Wang et al.'s bound gives $\chi(\Sigma) \geq 1.708$, whereas for $m=8$, Corollary \ref{Chroma_Treom} gives $\chi(\Sigma) \geq 2.07$. The chromatic number $\chi(\Sigma)$ actually equals $4$. 
	\end{example}
	\begin{figure}
		\centering
		\includegraphics[width=0.5\linewidth]{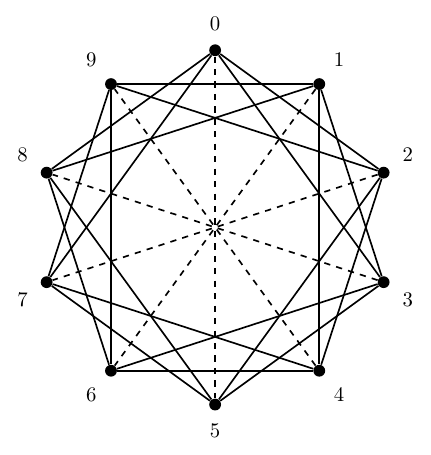}
		\caption{$\Sigma=(C_{10}\{2,3,5\}, \sigma)$, $\sigma$ as defined in Example \ref{circulant_exam}.} The dashed lines are negative edges.
		\label{fig:circulant}
	\end{figure}
	A natural question to ask is whether a result analogous to Theorem \ref{encompassing} can be proved under certain restrictions. Since the result is true in case of unsigned graphs, an obvious guess is to restrict to signed graphs with $\lambda_1 \geq \vert \lambda_n \vert$. However, the following example rules this out.
	
	\begin{example}
		Consider
		\[
		\Sigma=C_4^+\mathbin{\dot\cup}K_3^-,
		\]
		where $C_4^+$ is an all-positive $4$-cycle and $K_3^-$ is an all-negative triangle.  Both components have signed chromatic number $2$, so $\chi(\Sigma)=2$.  Their spectra are
		\[
		\spec(C_4^+)=\{2,0,0,-2\},
		\qquad
		\spec(K_3^-)=\{1,1,-2\}.
		\]
		Thus
		\[
		\spec(\Sigma)=\{2,1,1,0,0,-2,-2\},
		\]
		and the following holds 
		\[
		\lambda_1(\Sigma)=2=-\lambda_7(\Sigma).
		\]
		For $m=5$,
		\[
		\sum_{i=1}^5\lambda_i=4,
		\qquad
		-\sum_{i=1}^5\lambda_{7-i+1}=3.
		\]
		The expected lower bound is therefore
		\[
		\chi(\Sigma)\geq1+\frac43=\frac73,
		\]
		contradicting $\chi(\Sigma)=2$.
	\end{example}
	
	The disconnected example above might suggest adding connectedness.  That repair also fails, even if the spectral hypothesis is strengthened to $\lambda_1>-\lambda_n$.
	
	Construct a connected signed graph $\Omega$ as follows.
	
	\begin{enumerate}
		\item Start with an all-positive complete tripartite graph $K_{1,4,4}$.  Denote its singleton part by $\{x\}$ and its two $4$-vertex parts by $Y$ and $Z$.
		\item Take three disjoint all-negative copies $H_1,H_2,H_3$ of $K_6$.
		\item In each $H_i$, choose a vertex $h_i$, and add the positive bridge $xh_i$.
	\end{enumerate}
	
	The graph is connected.  It has a positive triangle, so $\chi(\Omega)\geq3$.  It also has a signed $3$-coloring with color set $\{-1,0,1\}$: assign
	\[
	c(x)=0,
	\qquad c(y)=1\ (y\in Y),
	\qquad c(z)=-1\ (z\in Z),
	\]
	and assign color $1$ to every vertex in each $H_i$.  Positive edges in $K_{1,4,4}$ join distinct colors, each positive bridge joins $0$ to $1$, and every negative edge in an $H_i$ joins two vertices colored $1$, which is allowed because $1\neq-1$.  Hence
	\[
	\chi(\Omega)=3.
	\]
	
	For the spectrum, use the equitable partition
	\[
	\{x\},\quad Y,\quad Z,\quad \{h_1,h_2,h_3\},
	\quad \bigcup_{i=1}^3\bigl(V(H_i)\setminus\{h_i\}\bigr).
	\]
	Its quotient matrix is
	\[
	Q=
	\begin{pmatrix}
		0&4&4&3&0\\
		1&0&4&0&0\\
		1&4&0&0&0\\
		1&0&0&0&-5\\
		0&0&0&-1&-4
	\end{pmatrix},
	\]
	with characteristic polynomial
	\[
	\det(tI-Q)=(t+2)(t+4)(t^3-2t^2-28t+44).
	\]
	The remaining eigenvalues can be accounted for explicitly.  Differences within $Y$ and within $Z$ give $0$ with multiplicity $3+3=6$.  For each $i$, vectors supported on $V(H_i)\setminus\{h_i\}$ and having coordinate sum zero give the eigenvalue $1$ with multiplicity $4$, hence multiplicity $12$ in total.  Let $(c_1,c_2,c_3)$ range over the two-dimensional subspace $c_1+c_2+c_3=0$.  Assigning value $-5c_i$ at $h_i$ and value $c_i$ at each of the other five vertices of $H_i$ gives two further independent $1$-eigenvectors.  Assigning the constant value $c_i$ on all six vertices of $H_i$ gives two independent $(-5)$-eigenvectors.  Therefore the remaining eigenvalues are
	\[
	0^{(6)},\qquad 1^{(14)},\qquad (-5)^{(2)}.
	\]
	The dimensions $6+14+2+5=27$ show that the quotient roots and these difference-space eigenvalues exhaust the spectrum.  Thus
	\[
	\chi_{A(\Omega)}(t)=
	t^6(t-1)^{14}(t+5)^2(t+2)(t+4)
	(t^3-2t^2-28t+44).
	\]
	Let the roots of
	\[
	f(t)=t^3-2t^2-28t+44
	\]
	be $\varepsilon<\beta<\alpha$.  The sign changes
	\[
	f(-6)<0<f(-5),\qquad
	f(1)>0>f(2),\qquad
	f(5)<0<f(6)
	\]
	show that
	\[
	-6<\varepsilon<-5,
	\qquad 1<\beta<2,
	\qquad 5<\alpha<6.
	\]
	These are all three roots of the cubic.  Therefore
	\[
	\lambda_1(\Omega)=\alpha,
	\qquad
	\lambda_{27}(\Omega)=\varepsilon.
	\]
	Since the root sum is
	\[
	\alpha+\beta+\varepsilon=2,
	\]
	one has
	\[
	\alpha+\varepsilon=2-\beta>0,
	\]
	so
	\[
	\lambda_1(\Omega)>-\lambda_{27}(\Omega).
	\]
	
	Take $m=23=27-4$.  The four largest eigenvalues are
	\[
	\alpha,\ \beta,\ 1,\ 1,
	\]
	and the four smallest are
	\[
	\varepsilon,\ -5,\ -5,\ -4.
	\]
	Using trace zero,
	\[
	P_{23}=14-\varepsilon,
	\qquad
	Q_{23}=\alpha+\beta+2.
	\]
	Writing $S=\alpha+\beta=2-\varepsilon$, one obtains
	\[
	\frac{P_{23}}{Q_{23}}
	=\frac{S+12}{S+2}.
	\]
	Because $\varepsilon>-6$, one has $S<8$, and hence
	\[
	\frac{S+12}{S+2}>2.
	\]
	Therefore Wocjan-Elphick type bound, if true, would give
	\[
	\chi(\Omega)>
	1+2=3,
	\]
	contradicting $\chi(\Omega)=3$.
	
	Numerically,
	\[
	\alpha\approx5.59780056,
	\quad \beta\approx1.53220648,
	\quad \varepsilon\approx-5.13000704,
	\]
	and the expected right side is approximately $3.09528941$.
	
	Thus
	connectedness does not rescue the proposed bound.
	
	The example also shows that replacing $\lambda_1\geq-\lambda_n$ by the strict inequality $\lambda_1>-\lambda_n$ is insufficient.

	\section{Signed Adjacency Eigenvalues}\label{Signed_Adjacency}
	
	Let $w_{k}^{+}(i,j)$ denote the number of positive walks of length $k$ from $i$ to $j$ and $w_{k}^{+}(i)$ denote the number of positive walks of length $k$ originating from $i$, and denote $w_{k}^{-}(i,j)$ and $w_{k}^{-}(i)$ analogously for the negative walks. Note that the $i,j$-th entry of $A^k$ is given by $w_{k}^{+}(i,j)-w_{k}^{-}(i,j)$ and hence $s_k(i):=(A^k\mathbf{1})_i= w_{k}^{+}(i)-w_{k}^{-}(i)$. Further, for $i \in V(G)$, let $N^1_i= N^+_i= \{j \in V(G) : \sigma(ij)=1\}, \quad N^{-1}_i=N^{-}_i= \{j \in V(G) : \sigma(ij)=-1\}$ and $\vert N^+_i \vert = d^+_i, \quad \vert N^{-}_i \vert = d^-_i$. Also define for each $i \in V(G)$, \[
	n_i=\sum_{j\sim i}
	\left(
	\vert N_i^{\sigma(ij)}\cap N_j\vert
	-
	\left | \vert N_i^{\sigma(ij)}\cap N_j^{\sigma(ij)}\vert
	-
	\vert N_i^{\sigma(ij)}\cap N_j^{-\sigma(ij)}\vert \right |
	\right).
	\]We also define two more quantities as follows: For $i \in V(G)$ with $d_i > 0$, let \\
	$$ m_i^+= \dfrac{1}{d_i}\bigg(\sum \limits_{j \in N_i^+}d_j^+ + \sum \limits_{j \in N_i^-}d_j^-\bigg)$$
	$$ m_i^-= \dfrac{1}{d_i}\bigg(\sum \limits_{j \in N_i^+}d_j^- + \sum \limits_{j \in N_i^-}d_j^+\bigg)$$ Define $m_i=m_i^++m_i^-$. It is easy to see that $m_i$ is just the average 2-degree of vertex $i$, that is, $m_i= \dfrac{\sum \limits_{j \sim i} d(j)}{d(i)}.$ In \cite{Stanic2019bounding}, Stani{\'c} gave an upper bound for the square of the largest eigenvalue of a signed graph in terms of $n_i,d_i,m_i$. We note down a simple proof for a strengthened bound below. For this, we define $\vert\vert \hspace{0.5mm} .\hspace{0.5mm}  \vert \vert_{\infty} : M_n(\mathbb{C) \rightarrow \mathbb{R}} $ by $\vert\vert  A  \vert \vert_{\infty}= \max \limits_{1\leq i \leq n}\quad\sum \limits_{j=1}^{n}\vert a_{ij}\vert$. It is well known (see \cite{MR2978290} Example 5.6.5 and Theorem 5.6.9) that $\vert\vert \hspace{0.5mm} .\hspace{0.5mm}  \vert \vert_{\infty}$ is a matrix norm and $\rho(A) \leq \vert \vert A  \vert \vert_{\infty}$.
	\begin{lemma}\label{norm_bound}
		Let $\Sigma$ be a signed graph with largest eigenvalue $\lambda_1$. Let $k$ be a positive integer. Then 
		$$\lambda_1^k \leq \vert \vert A^k \vert \vert_{\infty}= \max \limits_{1\leq i \leq n}\quad\sum \limits_{j=1}^{n}\vert w_k^+(i,j)-w_k^-(i,j)\vert$$
	\end{lemma}
	\begin{proof}
		$$\lambda_1^k \leq \rho(A)^k= \rho(A^k) \leq \vert \vert A^k \vert \vert_{\infty}=\max \limits_{1\leq i \leq n}\quad\sum \limits_{j=1}^{n}\vert w_k^+(i,j)-w_k^-(i,j)\vert$$
	\end{proof}
	Now define, \[
	p_{ij}
	=
	\left|N_i^{\sigma(ij)}\cap N_j^{\sigma(ij)}\right|,
	\qquad
	q_{ij}
	=
	\left|N_i^{\sigma(ij)}\cap N_j^{-\sigma(ij)}\right|.
	\]
	Note that $\left|N_i^{\sigma(ij)}\cap N_j\right|= p_{ij}+q_{ij}$. \\
	Stani\'{c}'s bound for the square of largest eigenvalue can be easily recovered from Lemma \ref{norm_bound}.
	\begin{theorem}
		Let $\Sigma$ be a signed graph with largest eigenvalue $\lambda_1$. 
		$$\lambda_1^2 \leq \max \{d_im_i-n_i : 1 \leq i \leq n \}$$
	\end{theorem}
	\begin{proof}
		Since $n_i= \sum \limits_{j \sim i}\left(p_{ij}+q_{ij}-\vert p_{ij}-q_{ij} \vert\right) $ and $a+b -\vert a - b \vert=2\min\{a,b\}$, we must have that, $n_i= 2\sum \limits_{j \sim i}\min\{p_{ij},q_{ij}\}$.
		Also,
		\begin{align*}
			\|A^2\|_{\infty}
			&=\max_{1\le i\le n}\left(\sum_{j=1}^{n}
			\left|w_2^+(i,j)-w_2^-(i,j)\right|\right)\\
			&=\max_{1\le i\le n}\left(\sum_{j=1}^{n}
			\left(
			w_2^+(i,j)+w_2^-(i,j)\right)
			-2 \sum_{j=1}^{n}\min\{w_2^+(i,j),\,w_2^-(i,j)\}\right)
			\\
		\end{align*}
		The quantities $p_{ij} \text{ and }q_{ij}$ count specific types of positive and negative $2$ walks from $i$ to $j$ respectively. Therefore $p_{ij} \leq w_2^+(i,j)$ and $q_{ij} \leq w_2^-(i,j)$. Therefore,\\
		\begin{align*}
			\|A^2\|_{\infty}
			&\leq \max_{1\le i\le n}\left(\sum_{j=1}^{n}
			\left(
			w_2^+(i,j)+w_2^-(i,j)\right)
			-2 \sum_{j \sim i}\min\{p_{ij},q_{ij}\}\right)
			\\
			&=\max_{1\le i\le n} (d_im_i-n_i)  
		\end{align*}
		The result now follows from Lemma \ref{norm_bound} with $k=2$.
	\end{proof}
	By a non-negative (positive) vector we mean the entries of that vector are non-negative (positive). In the same paper \cite{Stanic2019bounding}, Stani{\'c} gave an upper bound for the largest eigenvalue of a signed graph having a non-negative eigenvector associated with $\lambda_1$, involving the number of positive walks and negative walks. Stani{\'c} proved the following bound in terms of the absolute value of difference of positive and negative $k$ walks originating from a vertex called the net $k$ walks of that vertex. 
	
	\begin{theorem}\emph{(\cite{Stanic2019bounding}, Theorem 3.4)}\label{Stanic_result}
		If a non-negative eigenvector is associated with the largest eigenvalue $\lambda_{1}(\Sigma)$ of a signed graph $\Sigma$, then for every positive integer $k$,
		\[\lambda_{1}^{k}(\Sigma)\leq \max\{|w_{k}^{+}(i)-w_{k}^{-}(i)| : 1\leq i \leq n\}=\max\{\vert s_k(i) \vert : 1 \leq i \leq n \}.\]
	\end{theorem}
	
	We establish a general weighted bound on the net walks which can be used to strengthen Stani{\'c}'s absolute value bound.\\
	Define  for an arbitrary positive vector $y$, $R^{k}_i(y)= \dfrac{(A^ky)_i}{y_i}=\dfrac{\sum \limits_{j=1}^{n}(w_k^+(i,j)-w_k^-(i,j))y_j}{y_i}$
	
	\begin{theorem}\label{weighted_bound}
		Let $\lambda_{1}(\Sigma)$ be the largest eigenvalue of a signed graph $\Sigma$ with a non-negative eigenvector $x$ associated with it. Let $S=supp(x)=\{u \in V(G) : x_u >0\}$. Then for $k \geq 1$ and an arbitrary positive vector $y$,
		\[\min\{R^{k}_i(y) : i \in S\}\leq \lambda_{1}^{k}(\Sigma)\leq \max\{R^{k}_i(y) : i \in S\}.\] 
	\end{theorem}
	\begin{proof}
		Let $x=({x_{1},x_{2}, \cdots, x_{n}})$ be a non-negative eigenvector corresponding to $\lambda_{1}(\Sigma)$. Let $y=(y_{1}, y_{2}, \cdots, y_{n})$ be an arbitrary positive vector. Define $z=(R^{k}_1(y),R^{k}_2(y),\cdots, R^{k}_n(y))$.
		
		It is easy to see that  $$ y^{T}A^{k}(\Sigma)x =\lambda_{1}^{k}(\Sigma)\sum_{i=1}^{n}x_{i}y_{i}.$$ 
		Note that $$x^{T}A^{k}(\Sigma)y=\sum_{i=1}^{n}x_{i}\Big(\sum_{j=1}^{n}(w_{k}^{+}(i,j)-w_{k}^{-}(i,j))y_{j}\Big)=\sum_{i=1}^{n}x_{i}y_{i}z_{i}.$$

		Since~ $y^{T}A^{k}(\Sigma)x=x^{T}A^{k}(\Sigma)y$, we get
		$$\lambda_{1}^{k}(\Sigma)\sum_{i=1}^{n}x_{i}y_{i}=\sum_{i=1}^{n}x_{i}y_{i}z_{i}.$$
		
		Let $m=\min_{i \in S}\{z_{i}\}$ and $ M=\max_{i \in S}\{z_{i}\}.$ Then, $$m\sum_{i=1}^{n}x_{i}y_{i} \leq \lambda_{1}^{k}(\Sigma)\sum_{i=1}^{n}x_{i}y_{i}\leq M\sum_{i=1}^{n}x_{i}y_{i}.$$ Thus, $$m\leq \lambda_{1}^{k}(\Sigma)\leq M.$$
		
	\end{proof}
	
	When the eigenvector corresponding to the largest eigenvalue is positive, Theorem \ref{weighted_bound} can be optimized exactly as follows:\\
	\begin{corollary}
		If $x $ is an eigenvector corresponding to $\lambda_1$ such that all the entries of $x$ are positive and $y$ is an arbitrary positive vector, then for any positive integer $k$, $$\lambda_1^k= \min \limits_{y > 0} \max \limits_{1 \leq i \leq n} R_i^{k}(y)= \max \limits_{y > 0} \min \limits_{1 \leq i \leq n} R_i^{k}(y)$$
	\end{corollary}
	\begin{proof}
		Let $y$ be an arbitrary positive vector. Then $\lambda_1^k$ lies between the maximum and minimum of  $\{R^{k}_i(y) : i \in [n]\}$. Further for $y=x$ these extrema are attained.
	\end{proof}
	We get a strengthening of Theorem \ref{Stanic_result} as a simple consequence of Theorem \ref{weighted_bound}.
	\begin{theorem}\label{improve_Stanic_result}
		Let $\lambda_{1}(\Sigma)$ be the largest eigenvalue of a signed graph $\Sigma$ with a non-negative eigenvector $x$ associated with it. Let $S=supp(x) = \{ u \in V(G): x_u > 0\}$. Then for $k \geq 1$
		\[\lambda_{1}^{k}(\Sigma)\leq \max\{s_k(i) : i \in S\}.\]
	\end{theorem}
	\begin{proof}
		Choose $y=(1,1, \cdots, 1)$, then \[R^{k}_i(y)=\frac{1}{y_{i}}\sum_{j=1}^{n}(w_{k}^{+}(i,j)-w_{k}^{-}(i,j))y_{j}=\sum_{j=1}^{n}(w_{k}^{+}(i,j)-w_{k}^{-}(i,j))=w_{k}^{+}(i)-w_{k}^{-}(i).\]
		Therefore, by Theorem \ref{weighted_bound}, we get $\lambda_{1}^{k}(\Sigma)\leq \max\{w_{k}^{+}(i)-w_{k}^{-}(i) : i \in S\}=\max\{s_k(i) : i \in S\}.$ 
	\end{proof}
	
	The \emph{net degree} of a vertex, denoted by $d_{i}^{\pm}$ is defined as the difference of positive and negative neighbors of $i$ that is, $d_{i}^{\pm}=d_{i}^{+}-d_{i}^{-}$. 
	From the above result, we get an upper bound for the largest eigenvalue and its square as follows.
	\begin{corollary}
		If a non-negative eigenvector $x$ is associated with the largest eigenvalue $\lambda_{1}(\Sigma)$ of a signed graph  $\Sigma$ with support $S$, then 
		\[\lambda_{1}(\Sigma)\leq \max\{d_{i}^{+}-d_{i}^{-} : i \in S\}.\]
		\[\lambda_{1}^2(\Sigma)\leq \max\{d_{i}(m_i^+-m_i^-) : i \in S \quad \text{and } i \text{ is non-isolated}\}.\]
	\end{corollary}
	\begin{proof}
		The first inequality follows by using Theorem \ref{improve_Stanic_result} for $k=1$ and the fact that $s_1(i) = d_i^+-d_i^-$. For the second inequality consider for a non-isolated vertex $i$,
		\begin{align*}
			s_2(i)&=(A^2\mathbf{1})_i\\
			&= w_{2}^{+}(i)-w_{2}^{-}(i)\\
			&=\sum \limits_{j \in N_i^+}(d_j^+-d_j^-) - \sum \limits_{j \in N_i^-}(d_j^+-d_j^-)\\
			&=d_i(m^+_i-m^-_i)
		\end{align*}
		Now using Theorem \ref{improve_Stanic_result} for $k=2$ gives the second inequality.
	\end{proof}
	
	For a signed graph without isolated vertices, Theorem \ref{weighted_bound} gives another useful consequence.
	\begin{corollary}
		Let $\Sigma$ be a signed graph with no isolated vertices. If a non-negative eigenvector $x$ is associated with the largest eigenvalue $\lambda_{1}(\Sigma)$ of a signed graph  $\Sigma$ with support $S$, then 
		\[\lambda_{1}(\Sigma)\leq \max \limits_{i \in S} \dfrac{\sum \limits_{j \in N_i^+}d_j-\sum \limits_{j \in N_i^-}d_j}{d_i}\]
	\end{corollary}
	\begin{proof}
		Substitute $y= (d_1, \dots, d_n)$ and $k=1$ in Theorem \ref{weighted_bound}.
	\end{proof}    
	
	Next, we provide examples illustrating that the bound established in Theorem \ref{improve_Stanic_result} is sharper than the corresponding bound given in Theorem \ref{Stanic_result}.
	
	\begin{example}
		Consider the signed graph $\Sigma=(K_{4}, -)\cup(K_{3},+)$. The adjacency matrix is given as follows: $$A(\Sigma)=\begin{bmatrix}
			0&-1&-1&-1&0&0 &0 \\
			-1&0&-1&-1&0&0 &0 \\
			-1&-1&0&-1&0&0 &0 \\
			-1&-1&-1&0&0&0 &0 \\
			0&0&0&0&0&1 &1 \\
			0&0&0&0&1&0 &1 \\
			0&0&0&0&1&1 &0 \\   
		\end{bmatrix}$$
		It is easy to verify that the eigenvalues of $A(\Sigma)$ are $2, 1, 1, 1, -1, -1,$ and $-3$. Hence, the largest eigenvalue of $A(\Sigma)$ is $\lambda_{1}(\Sigma)=2$, with a corresponding nonnegative eigenvector $x=(0,0,0,0,1,1,1)$. Note that $\max\{d_{i}^{+}-d_{i}^{-}:i \in supp(x)\}=2$, whereas $\max\{|d_{i}^{+}-d_{i}^{-}|:1\leq i\leq 7\}=3$. Therefore, Theorem \ref{improve_Stanic_result} yields $\lambda_{1}(\Sigma)\leq 2$, while Theorem \ref{Stanic_result} gives the bound $\lambda_{1}(\Sigma)\leq 3$.
		
		Similarly, $\max\{w_{3}^{+}(i)-w_{3}^{-}(i):i \in supp(x)\}=8$, whereas $\max\{|w_{3}^{+}(i)-w_{3}^{-}(i)|:1\leq i\leq 7\}=27$. Consequently, Theorem \ref{improve_Stanic_result} implies $\lambda_{1}^{3}(\Sigma)\leq 8$, whereas Theorem \ref{Stanic_result} yields the estimate $\lambda_{1}^{3}(\Sigma)\leq 27$. Thus, the bounds obtained from Theorem \ref{improve_Stanic_result} are attained and are strictly sharper than those provided by Theorem \ref{Stanic_result}. 
	\end{example} 
	
	\begin{example}
		\begin{figure}
			\centering
			\includegraphics[width=0.4\linewidth]{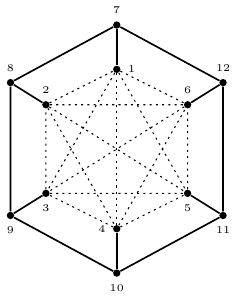}
			\caption{Connected Signed Graph (The negative edges are dashed lines)}
			\label{fig:Connected signed graph}
		\end{figure}
		Consider the signed graph $\Sigma$ in Figure \ref{fig:Connected signed graph}. The adjacency matrix is given as follows: $$A(\Sigma)=\begin{bmatrix}
			-A(K_{6})&I\\
			I&A(C_{6}) \\  
		\end{bmatrix}.$$
		The eigenvalues of $A(\Sigma)$ are $\frac{-3+\sqrt{53}}{2}, 2,2,\sqrt{2},\sqrt{2}, \frac{-1+\sqrt{13}}{2},0,0, -\sqrt{2}, -\sqrt{2}, \frac{-1-\sqrt{13}}{2}$ and $ \frac{-3-\sqrt{53}}{2}.$ Hence, the largest eigenvalue of $A(\Sigma)$ is $\lambda_{1}(\Sigma)=\frac{-3+\sqrt{53}}{2}$, with a corresponding nonnegative eigenvector $x=(2,2,2,2,2,2,7+\sqrt{53},7+\sqrt{53},7+\sqrt{53},7+\sqrt{53},7+\sqrt{53},7+\sqrt{53})$.
		Note that $\max\{d_{i}^{+}-d_{i}^{-}: i \in supp(x)\}=3$, whereas $\max\{|d_{i}^{+}-d_{i}^{-}|:1\leq i\leq 12\}=4$. Therefore, Theorem \ref{improve_Stanic_result} yields $\lambda_{1}(\Sigma)\leq 3$, while Theorem \ref{Stanic_result} gives the bound $\lambda_{1}(\Sigma)\leq 4$.
		
		Similarly, $\max\{w_{3}^{+}(i)-w_{3}^{-}(i):i \in supp(x)\}=27$, whereas $\max\{|w_{3}^{+}(i)-w_{3}^{-}(i)|:1\leq i\leq 12\}=113$. Consequently, Theorem \ref{improve_Stanic_result} implies $\lambda_{1}^{3}(\Sigma)\leq 27$, whereas Theorem \ref{Stanic_result} yields the estimate $\lambda_{1}^{3}(\Sigma)\leq 113$. Thus, the bounds obtained from Theorem \ref{improve_Stanic_result} are strictly sharper than those provided by Theorem \ref{Stanic_result}.
		
	\end{example}

	Next we switch our discussion to the least eigenvalue $\lambda_n(\Sigma)$ of a signed graph $\Sigma$. We note down a simple technique to derive lower bounds for $\lambda_n(\Sigma)$ using the Weyl inequality. First, we recall the Weyl inequality. Let $\lambda_{1}(M)\geq \lambda_{2}(M)\geq \cdots \geq \lambda_{n}(M)$ be the ordering of the eigenvalues of a symmetric matrix $M$. Then the eigenvalues of the sum of two symmetric matrices satisfy the following inequality. 
	\begin{lemma}\emph{(\cite{MR2978290}, Weyl inequality)}\label{weyl_ineq}
		Let $A, B \in M_{n}$ be symmetric matrices. Then  
		\[\lambda_{n}(A)+\lambda_{n}(B)\leq \lambda_{n}(A+B).\]
		Equality holds if and only if there is a nonzero vector $x$ such that $Ax=\lambda_{n}(A)x$, $Bx=\lambda_{n}(B)x$ and $(A+B)x=\lambda_{n}(A+B)x$.
	\end{lemma}
	
	The adjacency matrix of $\Sigma$ admits the decomposition
	\[
	A(\Sigma)=A(\Sigma_{+})+A(\Sigma_{-}),
	\]
	where $A(\Sigma_{+})$ and $A(\Sigma_{-})$ are the adjacency matrices of the positive and negative spanning subgraphs of $\Sigma$, respectively.
	By Lemma \ref{weyl_ineq}, 
	\begin{align} \label{inequ_adj_one}
		\nonumber \lambda_{n}(\Sigma)&=\lambda_{n}(A(\Sigma_{+})+A(\Sigma_{-}))\\
		\nonumber   &\geq \lambda_{n}(\Sigma_{+})+\lambda_{n}(\Sigma_{-}) \\
		&= \lambda_{n}(\Sigma_{+})-\lambda_{1}(-\Sigma_{-}).
	\end{align}
	Substituting lower bounds for $\lambda_n(\Sigma_+)$ and $-\lambda_1(-\Sigma_-)$ gives lower bound for the least eigenvalue of $\Sigma$. We illustrate one such example and characterize the graphs that attain the equality in the bound.
	
	Nikiforov proved the following result in \cite{nikiforov2002some}.
	\begin{theorem}\emph{(\cite{nikiforov2002some}, Theorem 1)}\label{Niki_largest}
		Let $G$ be a graph with $n$ vertices, $m$ edges and clique number $\omega(G).$ Let $\lambda_{1}(G)$ be the largest eigenvalue of the adjacency matrix $A(G).$ Then \[\lambda_{1}^{2}(G)\leq 2m \frac{\omega(G)-1}{\omega(G)}.\]
	\end{theorem}
	In the same paper \cite{nikiforov2002some}, Nikiforov gave another spectral radius bound as follows.
	\begin{theorem}\emph{(\cite{nikiforov2002some}, Theorem 7)}\label{nikiforov_bound}
		Let $G$ be a graph with $n$ vertices, $m$ edges and minimum degree $\delta$. Let $\lambda_{1}(G)$ be the largest eigenvalue of the adjacency matrix $A(G).$ Then \[\lambda_{1}(G)\leq \frac{\delta-1}{2} + \sqrt{2m-n\delta+\frac{(1+\delta)^2}{4}}.\]
	\end{theorem}
	Stanley gave the following spectral radius bound \cite{stanley1987bound} in terms of the number of edges.
	\begin{theorem}\emph{\cite{stanley1987bound}}\label{stanley_bound}
		Let $G$ be a graph with $n$ vertices and $m$ edges. Let $\lambda_{1}(G)$ be the largest eigenvalue of the adjacency matrix $A(G).$ Then \[\lambda_{1}(G)\leq -\frac{1}{2} + \sqrt{2m+\frac{1}{4}}.\]
	\end{theorem}
	
	For a signed graph $\Sigma$, let $\epsilon(\Sigma)$ denote the minimum number of edges whose removal causes $\Sigma$ to be balanced, called the frustration index of $\Sigma$. Let $\omega_b(\Sigma)$ be the maximum order of a balanced complete signed subgraph of $\Sigma$, called the balanced clique number of $\Sigma$. In 2023 \cite{kannan2023signed}, the following inequality was proved by Kannan and Pragada for signed graphs.
	\begin{lemma} \emph{(\cite{kannan2023signed}, Theorem 3.3)}\label{KannanShiva_BoundThm}
		Let $\Sigma$ be a signed graph with $m$ edges, $\epsilon(\Sigma)$ be the frustration index and $\omega_b(\Sigma)$ be the balanced clique number. Then
		\begin{equation*}
			\lambda_1^2(\Sigma) \leq 2(m-\epsilon(\Sigma)) \left(1-\dfrac{1}{\omega_b(\Sigma)}\right)
		\end{equation*}
	\end{lemma}
	The equality cases in the above result were characterized by Lan et al..
	\begin{lemma}\emph{(\cite{lan2023remarks}, Theorem 1.7)}\label{KS_Bound_Equal}
		Let $\Sigma$ be a connected signed graph with $m \geq 1$ edges, $\epsilon(\Sigma)$ be the frustration index and $\omega_b(\Sigma)$ be the balanced clique number. Then
		\begin{equation}\label{KannanShiva_Bound}
			\lambda_1^2(\Sigma) = 2(m-\epsilon(\Sigma)) \left(1-\dfrac{1}{\omega_b(\Sigma)}\right)
		\end{equation}
		holds if and only if $\Sigma$ satisfies one of the following:
		\begin{itemize}

			\item $\omega_b(\Sigma)>2$ and 
			$\Sigma$ is a balanced regular complete $\omega_b(\Sigma)$-partite graph; 
			
			\item $\omega_b(\Sigma)=2$ and
			$\Sigma$ is a balanced complete bipartite graph.
		\end{itemize}
	\end{lemma}
	The edge bipartiteness of a graph $G$, denoted by $\epsilon_b(G)$, is the minimum number of edges whose removal makes the graph bipartite. Using bipartiteness, Kannan and Pragada gave a modified version of a known classical bound on the least eigenvalue (see \cite{favaron1993some}) of the adjacency matrix of an unsigned graph $G$. We include Kannan and Pragada's proof in order to characterize the graphs attaining the equality.
	
	\begin{lemma}\emph{(\cite{kannan2023signed}, Corollary 3.4)}\label{smallest_eigenvalue}
		Let $G$ be a graph on $n$ vertices and $m \geq 1$ edges. Let $\epsilon_{b}(G)$ denote the edge bipartiteness and $\lambda_{n}(G)$ be the least eigenvalue of the adjacency matrix $A(G)$ of the graph $G.$ Then
		\begin{equation}\label{Favaron_Ineq}
			\lambda_{n}^{2}(G)\leq m-\epsilon_{b}(G)
		\end{equation}
	\end{lemma}
	\begin{proof}
		Set $\Sigma=(G,-)$ that is, the underlying graph $G$ with all edges given negative sign. By Lemma \ref{KannanShiva_BoundThm}, we get,\\
		\begin{equation*}
			\lambda_n^2(G) =\lambda_1^2(\Sigma)\leq  2(m-\epsilon(\Sigma)) \left(1-\dfrac{1}{\omega_b(\Sigma)}\right) 
		\end{equation*}
		Now in an all negative edged signed graph every triangle is unbalanced, hence $\omega_b(\Sigma)=2$. Moreover, in such a graph every even cycle is balanced, and every odd cycle is unbalanced. Therefore $\Sigma$ is balanced if and only if $G$ contains no odd cycles, that is, $G$ is bipartite. This gives that $\epsilon(\Sigma)=\epsilon_b(G)$. Substituting in the above inequality, we get the desired result.
	\end{proof}
	Lemma \ref{KS_Bound_Equal} can be used to prove the equality cases in Lemma \ref{smallest_eigenvalue}.
	\begin{lemma}\label{Fav_Equality}
		For a graph $G$ having at least one edge, the equality
		\begin{equation}\label{Fav_Eq}
			\lambda_{n}^{2}(G)= m-\epsilon_{b}(G)
		\end{equation}
		holds if and only if $G$ is isomorphic to a complete bipartite graph with at least one edge and possibly some isolated vertices.
	\end{lemma}
	\begin{proof}
		The sufficiency part is obvious.\\
		For the necessity part, suppose the equality \eqref{Fav_Eq} holds for a graph $G$. Let $C_1,\cdots, C_k$ be the connected components of $G$ with $m_1, \cdots, m_k$ edges respectively and having edge bipartiteness $\epsilon_b(C_1), \cdots,\epsilon_b(C_k)$ respectively.\\
		Without loss of generality, let $C_1$ be the component of $G$ such that $\lambda_{min}(C_1)=\lambda_n(G)$.
		Applying Lemma \ref{smallest_eigenvalue} on $C_1$ we get $ \lambda_n^2(G)=\lambda_{min}^2(C_1)\leq m_1 - \epsilon_b(C_1) \leq m-\epsilon_b(G)$. But since \eqref{Fav_Eq} holds, we must have $\lambda_{min}^2(C_1)=m_1 - \epsilon_b(C_1)$. From the proof of Lemma \ref{smallest_eigenvalue}, it is clear that 
		this equality holds if and only if $\lambda_1^2(\Sigma_1) = 2(m_1-\epsilon(\Sigma_1))\left(1-\dfrac{1}{\omega_b(\Sigma_1)}\right)$ where $\Sigma_1=(C_1,-)$. By Lemma \ref{KS_Bound_Equal} and the fact that $G$ has at least one edge, this equality holds if and only if $\Sigma_1$ is a balanced complete bipartite graph and hence $C_1$ must be a complete bipartite graph.
		
		On the other hand, again since the equality \eqref{Fav_Eq} holds, we must also have $m_1-\epsilon_b(C_1) = m-\epsilon_b(G)$ But, $m-\epsilon_b(G) =\sum \limits_{i=1}^{k} (m_i-\epsilon_b(C_i))$. Thus we must have $\sum \limits_{i\neq 1} (m_i-\epsilon_b(C_i))=0$ which forces $ (m_i-\epsilon_b(C_i))=0$ for all $i=2,3,\cdots, k$. This happens only if $C_2, \cdots, C_k$ are all isolated vertices. Therefore, $G$ must be a complete bipartite graph with possibly some isolated vertices.
	\end{proof}
	
	We are now ready to illustrate a use of equation \eqref{inequ_adj_one} to provide a sharp bound for the least eigenvalue of the adjacency matrix of a signed graph.\\
	The following result gives a lower bound of $\lambda_{n}(\Sigma)$ involving the number of positive edges, negative edges and edge bipartiteness of the subgraph $\Sigma_+$ and the minimum degree of the subgraph $\Sigma_-$ of  $\Sigma$.
	\begin{theorem}\label{adj_small_three}
		Let $\Sigma$ be a signed graph with spanning subgraph $\Sigma_{+}$ comprising the $m^{+}$ positive edges and spanning subgraph $\Sigma_{-}$ comprising the $m^{-}$ negative edges. Let $\epsilon_{b}(\Sigma_{+})$ be the edge bipartiteness of $\Sigma_{+}$ and $\delta_{-}$ be the minimum degree of the subgraph $-\Sigma_{-}.$ Then
		\begin{equation}\label{Ineq_three}
			\lambda_{n}(\Sigma)\geq  -\sqrt{m^{+}-\epsilon_{b}(\Sigma_{+})}-\frac{\delta_{-}-1}{2}-\sqrt{2m^{-}-n\delta_{-}+\frac{(1+\delta_{-})^{2}}{4}}.
		\end{equation}
		If $m^+ \geq 1$, equality in \eqref{Ineq_three} holds if and only if the vertex set $V(\Sigma)$ can be partitioned as $V(\Sigma)= S_1 \dot{\cup} S_2 \dot{\cup} Z $ with $S_1, S_2 \neq \phi$ such that the following conditions hold:
		\begin{itemize}
			\item [(C1)] $\Sigma_+$ is the disjoint union of a complete bipartite graph $K_{a,b}$, with parts $S_1, S_2$ and $t = \vert Z \vert$ isolated vertices.
			\item [(C2)] $-\Sigma_-$ is $\delta_-$ regular.
			\item [(C3)] No negative edges join $Z$ to $S_1 \cup S_2$.
		\end{itemize}
	\end{theorem}
	\begin{proof}
		
		Applying Lemma \ref{smallest_eigenvalue} to the graph $\Sigma_{+}$ and Theorem \ref{nikiforov_bound} to the graph $-\Sigma_{-}$, we obtain 
		\begin{equation}\label{Fav_Sigma+}
			\lambda_{n}(\Sigma_{+})\geq -\sqrt{m^{+}-\epsilon_{b}(\Sigma_{+})}
		\end{equation}
		and 
		\begin{equation}\label{Niki_Bound_}
			-\lambda_{1}(-\Sigma_{-})\geq -\frac{\delta_{-}-1}{2} - \sqrt{2m^{-}-n\delta_{-}+\frac{(1+\delta_{-})^2}{4}}.
		\end{equation}
		Substituting these in \eqref{inequ_adj_one}, we get
		\[\lambda_{n}(\Sigma)\geq  -\sqrt{m^{+}-\epsilon_{b}(\Sigma_{+})}-\frac{\delta_{-}-1}{2}-\sqrt{2m^{-}-n\delta_{-}+\frac{(1+\delta_{-})^{2}}{4}}.\]
		
		Now we shall prove the equality case. \\
		Suppose $(C1), (C2), (C3)$ hold. Define $x \in \mathbb{R}^n$ as follows:\\
		$$
		x_w =
		\begin{cases}
			\frac{1}{\sqrt{2a}}, & \text{if } w \in S_1,\\
			\dfrac{-1}{\sqrt{2b}}, & \text{if } w \in S_2,\\
			0, & \text{if } w\in Z. \\
		\end{cases}
		$$
		Fix $u \in S_1$. The positive neighbors of $u$ are the $b$ vertices in $S_2$. By $(C1)$, the negative neighbors of $u$ cannot be in $S_2$ and by $(C3)$ those cannot be in $Z$ as well. By $(C2)$ there are exactly $\delta_-$ negative neighbors of $u$ all of which should lie in $S_1$. Therefore, we have,
		\begin{equation*}
			(Ax)_u= \sum \limits_{v \in N^+_i}x_v -  \sum \limits_{v \in N^-_i}x_v=\dfrac{-b}{\sqrt{2b}}-\dfrac{\delta_-}{\sqrt{2a}}=(-\sqrt{ab}-\delta_-)x_u
		\end{equation*}
		Symmetrically, we can show that for $u \in S_2$
		\begin{equation*}
			(Ax)_u=(-\sqrt{ab}-\delta_-)x_u
		\end{equation*}
		Also for $u \in Z$ using $(C1)$ and $(C3)$ we easily get that
		\begin{equation*}
			(Ax)_u=(-\sqrt{ab}-\delta_-)x_u=0
		\end{equation*}
		Therefore $(-\sqrt{ab}-\delta_-)$ is an eigenvalue of $\Sigma$ which forces $\lambda_n(\Sigma) \leq (-\sqrt{ab}-\delta_-)$.\\
		On the other hand, by $(C1)$
		, $\epsilon_b(\Sigma_+)=0, m^+=ab$ and by $(C2)$, $2m^-=n\delta_-$. Therefore, inequality \eqref{Ineq_three} yields,
		\begin{align*}
			\lambda_n(\Sigma) &\geq -\sqrt{ab} -\dfrac{\delta_--1}{2}-\sqrt{\dfrac{(1+\delta_-)^2}{4}}\\
			&= -\sqrt{ab} -\dfrac{\delta_--1}{2}-\dfrac{(1+\delta_-)}{2}=(-\sqrt{ab}-\delta_-)
		\end{align*}
		This forces $\lambda_n(\Sigma) = (-\sqrt{ab}-\delta_-)$ and the equality in \eqref{Ineq_three} holds.\\
		Conversely, suppose \eqref{Ineq_three} is an equality. Then equality must hold in \eqref{inequ_adj_one}, \eqref{Fav_Sigma+}, \eqref{Niki_Bound_}.
		
		Since $m^+ \geq 1$ and \eqref{Fav_Sigma+} holds, by Lemma \ref{Fav_Equality}, we get that $\Sigma_+ \cong K_{a,b} \dot{\cup} tK_1$ with $a,b \geq 1, t \geq 0$. Write $S_1, S_2$ for the parts forming the complete bipartite structure and call the set of $t$ isolated vertices as $Z$. Then $\vert S_1 \vert =a, \vert S_2 \vert =b$. This is $(C1)$. 
		
		Clearly $\lambda_n(\Sigma_+) = -\sqrt{ab}$ and is a simple eigenvalue with eigenvector $x$ as defined above.\\
		Since \eqref{inequ_adj_one} holds with equality by Lemma \ref{weyl_ineq}, there is a unit eigenvector $y$ with $A(\Sigma_+) y =-\sqrt{ab}. y$ and $A(-\Sigma_-)y = \rho y$ where $\rho=\lambda_1(-\Sigma_-)$. By simplicity of $\lambda_n(\Sigma_+)$, $y =\pm x$, so $x$ itself is an eigenvector of $A(-\Sigma_-)$ corresponding to $\rho$.
		
		Let $C$ be a component of $-\Sigma_-$ satisfying $C\cap(S_1 \cup S_2) \neq \phi$. Let $x_{|C}$ be the vector $x$ restricted to vertices in $C$. Then it is easy to see that, $x_{|C} \neq 0$ and $A(C) x_{|C}= \rho x_{|C} $. This gives $\rho \leq \rho(C)$ and since $C$ is a component of $-\Sigma_-$ we must have that $\rho(C) \leq \rho$. This forces $\rho= \rho(C)$. Using Perron Frobenius Theorem on the connected graph $C$, we get that all entries of $x_{|C}$ must be of same sign and non-zero. This forces $C \cap Z = \phi$ which is $(C3)$. Moreover, the condition that the entries of $x_{|C}$ are of the same sign forces that $C$ must completely lie in $S_1$ or in $S_2$. By the structure of $x$ in $S_1$ and $S_2$, $x$ is forced to be constant vector and hence $A(C) \mathbf{1}=\rho \mathbf{1}$ which implies that $C$ is a $\rho$- regular graph. Thus, we have proved that for a component $C$ of $-\Sigma_-$ with $C \cap (S_1 \cup S_2)$, it must be a $\rho$- regular graph completely lying in $S_1$ or $S_2$.
		
		Let $u \in S_1$ . If $C$ is the component of $-\Sigma_-$ containing $u$ then by the previous argument $C$ must be contained completely in $S_1$ and $C$ is $\rho$-regular. Therefore $\rho+1 \leq \vert S_1 \vert =a$. Similarly we can argue that $\rho+1 \leq b$. These inequalities give that 
		\begin{equation}\label{contradic}
			a+b \geq 2\rho+2
		\end{equation}
		\textbf{Claim:} $\delta_-= \rho$ \\
		\textit{Proof:} Clearly $\rho \geq \dfrac{2m^-}{n} \geq \delta_-$. Suppose $\delta_- \neq \rho$ and let $d:= \delta_-<\rho.$\\
		Since $\forall u \in S_1 \cup S_2$, $d_{-\Sigma_-}(u) =\rho$, $d$ is attained by a vertex in $Z$ and hence $Z \neq \phi$. Since every vertex in $Z$ has degree at least $d$ in $-\Sigma_-$, setting $m_Z:= \vert E(-\Sigma_-[Z])\vert$, we get that, $2m_Z\geq \vert Z \vert d$. Therefore, counting negative edges in $\Sigma$, we get $2m^-=(a+b)\rho+2m_Z$. Now since \eqref{Niki_Bound_} holds with equality we must have , 
		\begin{align*}
			\rho = \dfrac{d-1}{2}+ &\sqrt{2m^{-}-nd +\dfrac{(d+1)^2}{4}}\\
			\rho-\dfrac{(d-1)}{2}&=\sqrt{2m^--nd +\dfrac{(d+1)^2}{4}}\\
			\rho^2+\dfrac{(d-1)^2}{4}-&\rho(d-1)= 2m^--nd+\dfrac{(1+d)^2}{4}\\
			\rho^2+\rho(1-d)+&\dfrac{(d-1)^2}{4}-\dfrac{(d+1)^2}{4}+nd=2m^-\\
			\rho^2+\rho(1-d)+&d(n-1)=2m^- \\
		\end{align*}
		
		With $n=a+b+\vert Z \vert$ we get,
		\begin{equation*}
			\rho^{2}+\rho-d\rho+(a+b)d+|Z|d-d
			=(a+b)\rho+2m_Z
			\ge (a+b)\rho+|Z|d.
		\end{equation*}
		Cancelling $\vert Z \vert d$ and rearranging we get,
		\begin{equation*}
			\rho^{2}+\rho-d\rho-d
			\ge (a+b)(\rho-d),
			\quad \text{i.e.,} \quad
			(\rho+1)(\rho-d)
			\ge (a+b)(\rho-d).
		\end{equation*}
		Since $\rho -d > 0$ we must have $(a+b) \leq \rho+1$ contradicting \eqref{contradic}. Therefore $\delta_-=\rho$.
		\newline
		Finally, every component $C\subseteq Z$ of $-\Sigma^{-}$ satisfies
		\[
		\rho \ge \rho(C) \ge \bar{d}(C) \ge \delta(C) \ge \delta^{-}=\rho,
		\]
		where $\bar{d}(C)$ is the average degree; so all these are equalities and
		$C$ is $\rho$-regular. Thus $-\Sigma^{-}$ is $\rho$-regular that is $\delta_-$-regular  throughout,
		which is (C2).
	\end{proof}
	
	In the following result, we establish another lower bound for the smallest eigenvalue $\lambda_n(\Sigma)$ of a signed graph $\Sigma$.

	\begin{theorem}\label{adj_small_one}
		Let $\Sigma$ be a signed graph with at least one negative edge and let $\Sigma_{+}$ and $\Sigma_{-}$ denote the spanning subgraphs consisting only of the positive and negative edges of $\Sigma$, respectively. Suppose that $\Sigma_{+}$ contains $m^{+}$ positive edges and $\Sigma_{-}$ contains $m^{-}$ negative edges. Let $\epsilon_{b}(\Sigma_{+})$ denote the edge bipartiteness of $\Sigma_{+}$, and let $\omega(-\Sigma_{-})$ denote the clique number of the unsigned graph $-\Sigma_{-}$. Then
		\[
		\lambda_{n}(\Sigma)\geq -\sqrt{m^{+}-\epsilon_{b}(\Sigma_{+})}
		-\sqrt{2m^{-}\left(1-\frac{1}{\omega(-\Sigma_{-})}\right)}.
		\]
		
		Equality holds if $\Sigma=(K_{n},-)$,  or if $\Sigma=(K_{p,q},-)$, where $p+q=n$.
		
	\end{theorem}
	\begin{proof}
		
		Applying Lemma \ref{smallest_eigenvalue} to the unsigned graph $\Sigma_{+}$ and Theorem \ref{Niki_largest} to the unsigned graph $-\Sigma_{-}$, we obtain $$\lambda_{n}(\Sigma_{+})\geq -\sqrt{m^{+}-\epsilon_{b}(\Sigma_{+})}$$
		and $$-\lambda_{1}(-\Sigma_{-})\geq -\sqrt{2m^{-}\Big(1-\frac{1}{\omega(-\Sigma_{-})}\Big)}.$$
		Substituting these estimates in the inequality \eqref{inequ_adj_one}, we obtain
		\[\lambda_{n}(\Sigma)\geq  -\sqrt{m^{+}-\epsilon_{b}(\Sigma_{+})}-\sqrt{2m^{-}\Big(1-\frac{1}{\omega(-\Sigma_{-})}\Big)}.\]
		
		If $\Sigma=(K_{n}, -)$ or $\Sigma=(K_{p,q}, -)$, then it is easy to see that the equality holds.
	\end{proof}
	
	\begin{corollary}\label{adj_small_two}
		Let $\Sigma=(G,\sigma)$ be a signed graph with at least one positive edge and at least one negative edge. Let $\Sigma_{+}$ and $\Sigma_{-}$ denote the spanning subgraphs of $\Sigma$ induced by the positive and negative edges, respectively, with $m^{+}$ and $m^{-}$ edges. Let $\omega(\Sigma_{+})$ and $\omega(-\Sigma_{-})$ denote the clique numbers of $\Sigma_{+}$ and $-\Sigma_{-}$, respectively. Then
		\[
		\lambda_{n}(\Sigma)\geq
		-\sqrt{2m^{+}\left(1-\frac{1}{\omega(\Sigma_{+})}\right)}
		-\sqrt{2m^{-}\left(1-\frac{1}{\omega(-\Sigma_{-})}\right)}.
		\]
		
	\end{corollary}
	\begin{proof}
		Note that since $\Sigma_+$ has at least one edge, $\omega(\Sigma_+) \geq 2$ which gives $2m^{+}\left(1-\frac{1}{\omega(\Sigma_{+})}\right) \geq m^+ \geq m^+ - \epsilon_b(\Sigma_+)$. Now the result follows from Theorem \ref{adj_small_one}.
	\end{proof}
	Next, we use Nikiforov's idea in Corollary \ref{adj_small_two} and get the following result.
	\begin{corollary}
		Let $\Sigma$ be a signed graph with $n$ vertices and $m$ edges. Let $m^{+}$ and $m^{-}$ denote the number of positive edges and negative edges in $\Sigma.$ Then
		\[\lambda_{n}(\Sigma)\geq  -\sqrt{2m^{+}\Bigg(1-\Big\lfloor\frac{1}{2}+\sqrt{2m^{+}+\frac{1}{4}}\Big\rfloor^{-1}\Bigg)}-\sqrt{2m^{-}\Bigg(1-\Big\lfloor\frac{1}{2}+\sqrt{2m^{-}+\frac{1}{4}}\Big\rfloor^{-1}\Bigg)}.\]
	\end{corollary}
	\begin{proof}
		Let $\omega(\Sigma_{+})$ be the clique number of signed subgraph $\Sigma_{+}$. Then \[\binom{\omega(\Sigma_{+})}{2}\leq m^{+}.\]  This inequality gives us \[\omega(\Sigma_{+})\leq \frac{1}{2}+\sqrt{2m^{+}+\frac{1}{4}}.\]
		Since $\omega(\Sigma_{+})$ is an integer, so \[\omega(\Sigma_{+})\leq \Big\lfloor\frac{1}{2}+\sqrt{2m^{+}+\frac{1}{4}}\Big\rfloor.\]
		Similarly, \[\omega(-\Sigma_{-})\leq \Big\lfloor\frac{1}{2}+\sqrt{2m^{-}+\frac{1}{4}}\Big\rfloor.\]
		Substituting these into Corollary \ref{adj_small_two}, we get the result.
	\end{proof}

	If $\Sigma$ is a signed graph, then we get the following corollaries using Stanley's bound \cite{stanley1987bound} and Hong's bound \cite{yuan1988bound} respectively.

	\begin{corollary}
		Let $\Sigma$ be a signed graph with subgraph $\Sigma_{+}$ consisting of all $m^{+}$ positive edges and subgraph $\Sigma_{-}$ consisting of all $m^{-}$ negative edges. Let $\epsilon_{b}(\Sigma_{+})$ be the edge bipartiteness of $\Sigma_{+}.$ Then
		\[\lambda_{n}(\Sigma)\geq \frac{1}{2} -\sqrt{m^{+}-\epsilon_{b}(\Sigma_{+})}-\sqrt{2m^{-}+\frac{1}{4}}.\]
		Equality holds if either $\Sigma=(K_{n}, -)$ or $\Sigma=(K_{\lfloor \frac{n}{2}\rfloor, \lceil\frac{n}{2}\rceil}, +).$
	\end{corollary}
	
	\begin{proof}
		By applying Lemma \ref{smallest_eigenvalue} on $\Sigma_{+}$, we get $\lambda_{n}(\Sigma_{+})\geq -\sqrt{m^{+}-\epsilon_{b}(\Sigma_{+})}.$
		By applying Theorem on $-\Sigma_{-}$ \ref{stanley_bound}, we get $-\lambda_{1}(-\Sigma_{-})\geq \frac{1}{2} -\sqrt{2m^{-}+\frac{1}{4}}.$
		By substituting these bounds in inequality (\ref{inequ_adj_one}), we get
		\[\lambda_{n}(\Sigma)\geq \frac{1}{2} -\sqrt{m^{+}-\epsilon_{b}(\Sigma_{+})}-\sqrt{2m^{-}+\frac{1}{4}}.\]
		
		If $\Sigma=(K_{n}, -)$ or $\Sigma=(K_{\lfloor \frac{n}{2}\rfloor, \lceil\frac{n}{2}\rceil}, +)$, then equality is obvious.
	\end{proof}
	
	\begin{corollary}
		Let $\Sigma$ be a signed graph with subgraph $\Sigma_{+}$ consisting of all $m^{+}$ positive edges and subgraph $\Sigma_{-}$ consisting of all $m^{-}$ negative edges. Let $\epsilon_{b}(\Sigma_{+})$ be the edge bipartiteness of $\Sigma_{+}.$ If the subgraph $\Sigma_{-}$ is connected, then
		\[\lambda_{n}(\Sigma)\geq -\sqrt{m^{+}-\epsilon_{b}(\Sigma_{+})}-\sqrt{2m^{-}-n+1}.\]
		Equality holds if $\Sigma=(K_{n}, -)$.
	\end{corollary}


	
	\section*{Acknowledgments}
	The first author acknowledges support from the DST INSPIRE Fellowship (IF220692). The second and third authors acknowledge financial support from the ANRF-CRG grant (File Number: CRG/2023/002747). The authors acknowledge the use of ChatGPT for improving the exposition of the manuscript, refining some of the proofs and constructing the example. All mathematical statements have been independently verified by the authors, who are solely responsible for the correctness of the results presented in this paper.
	
	\section*{Data Availability} No data were generated or analysed in this study, as the work is entirely theoretical. 
	
	\section*{Conflict of interest:} 
	The authors state that they possess no conflicts of interest.
	\bibliographystyle{amsplain}
	\bibliography{Paper1ref}
	
	\vspace{0.4cm}
	\affl{Abhay Jayarajan}{ma23resch02001@iith.ac.in, abhayjayarajan@gmail.com}{ Department of Mathematics, Indian Institute of Technology Hyderabad, Kandi, Sangareddy 502284, India.}
	
	\affl{M. Rajesh Kannan}{rajeshkannan@math.iith.ac.in, rajeshkannan1.m@gmail.com}{Department of Mathematics, Indian Institute of Technology Hyderabad, Kandi, Sangareddy 502284, India.}
	
	\affl{Priti Prasanna Mondal}{priti.prasanna@prjt.math.iith.ac.in, pritiprasanna1992@gmail.com}{Department of Mathematics, Indian Institute of Technology Hyderabad, Kandi, Sangareddy 502284, India.}
	
	\affl{Shivaramakrishna Pragada}{shivaramakrishna\_pragada@sfu.ca, shivaramkratos@gmail.com}{Department of Mathematics, Simon Fraser University, Burnaby, Canada}
\end{document}